\documentclass[11pt]{amsart}
\usepackage{hyperref}
\usepackage{amsfonts}
\usepackage{amsmath}
\usepackage{amssymb}
\usepackage{amscd}
\usepackage{graphicx}
\usepackage{enumitem}
\usepackage{latexsym}
\usepackage{color}
\usepackage{epsfig}
\usepackage{amsopn}
\usepackage{xypic}
\usepackage{mathrsfs}
\usepackage{array}

\usepackage{booktabs}
\usepackage{longtable}
\theoremstyle{plain}
\newtheorem{lemma}{Lemma}[section]
\newtheorem*{theorem*}{Theorem}
\newtheorem*{lemma*}{Lemma}
\newtheorem*{proposition*}{Proposition}
\newtheorem*{conjecture*}{Conjecture}
\newtheorem*{corollary*}{Corollary}
\newtheorem*{problem*}{Problem}
\newtheorem{theorem}[lemma]{Theorem}

\newtheorem{corollary}[lemma]{Corollary}
\newtheorem{proposition}[lemma]{Proposition}

\newtheorem{problem}[lemma]{Problem}

\theoremstyle{definition}
\newtheorem{definition}[lemma]{Definition}
\newtheorem{example}[lemma]{Example}
\newtheorem{remark}[lemma]{Remark}

\newcommand{\F}[1]{\mathscr{#1}}
\newcommand{\fto}[1]{\stackrel{#1}{\to}}

\newcommand{\Z}{\mathbb{Z}}

\newcommand{\FF}{\mathbb{F}}
\renewcommand{\F}{\mathbb{F}}
\newcommand{\CC}{\mathbb{C}}
\newcommand{\Q}{\mathbb{Q}}
\newcommand{\R}{\mathbb{R}}

\newcommand{\OO}{\mathcal{O}}
\newcommand{\te}{\otimes}

\newcommand{\cF}{\mathcal F}

\newcommand{\cM}{\mathcal M}

\newcommand{\cP}{\mathcal P}
\newcommand{\cT}{\mathcal{T}}
\newcommand{\cE}{\mathcal{E}}
\newcommand{\cV}{\mathcal{V}}
\newcommand{\cW}{\mathcal{W}}

\newcommand{\ZZ}{\mathbb{Z}}

\renewcommand{\P}{\mathbb{P}}

\newcommand{\PP}{\mathbb{P}}

\DeclareMathOperator{\ch}{ch}

\DeclareMathOperator{\Hom}{Hom}

\DeclareMathOperator{\Pic}{Pic}
\DeclareMathOperator{\Sym}{Sym}

\DeclareMathOperator{\Ext}{Ext}

\DeclareMathOperator{\sHom}{\mathcal{H} \textit{om}}

\DeclareMathOperator{\SL}{SL}

\begin{document}

\date{\today}
\author[I. Coskun]{Izzet Coskun}
\address{Department of Mathematics, Statistics and CS \\University of Illinois at Chicago, Chicago, IL 60607}
\email{coskun@math.uic.edu}
\author[J. Huizenga]{Jack Huizenga}
\address{Department of Mathematics, The Pennsylvania State University, University Park, PA 16802}
\email{huizenga@psu.edu}

\thanks{During the preparation of this article the first author was partially supported by the  NSF grant DMS-1500031 and  the second author was partially supported  by the NSA\ Young Investigator Grant H98230-16-1-0306}

\title[Brill-Noether Theorems and Gaeta resolutions]{Brill-Noether Theorems  and globally generated vector bundles on Hirzebruch surfaces}

\begin{abstract}
In this paper, we show that the cohomology of a general stable bundle on a Hirzebruch surface is determined by the Euler characteristic provided that the first Chern class satisfies necessary intersection conditions. More generally, we compute the Betti numbers of a general stable bundle.  We also show that a general stable bundle on a Hirzebruch surface has a special resolution generalizing the Gaeta resolution on the projective plane. As a consequence of these results, we classify Chern characters such that the general stable bundle is globally generated. 
\end{abstract}

\maketitle

\setcounter{tocdepth}{1}
\tableofcontents

\section{Introduction}
The Brill-Noether theorem of G\"{o}ttsche and Hirschowitz \cite{GottscheHirschowitz} shows that a general stable bundle on $\PP^2$ has at most one nonzero cohomology group. On a Hirzerbruch surface $\F_e$ the situation is not so simple---the section of negative self-intersection can cause every bundle with given numerical invariants to have interesting cohomology.  In this paper,  we determine necessary and sufficient conditions on numerical invariants  which ensure that the general stable bundle on $\F_e$ has at most one nonzero cohomology group.  Essentially equivalently, we also compute the Betti numbers of a general stable bundle.

We then show that the general stable sheaf on $\F_e$ has a special resolution by direct sums of line bundles.  These resolutions generalize the Gaeta resolution of a general sheaf on $\P^2$, and can be viewed as giving unirational parameterizations of moduli spaces of sheaves \cite{Eisenbudsyzygies, Gaeta}.  Thus, these resolutions are a convenient tool for describing a general sheaf.  As a consequence of the Brill-Noether theorem and the Gaeta-type resolution, we completely determine when a general stable bundle on a Hirzebruch surface is globally generated.  The case of $\F_1$ implies an analogous result for $\P^2$ which sharpens a theorem of Bertram, Goller and Johnson \cite{BGJ}.  These theorems play crucial roles in the construction of theta and Brill-Noether divisors and in the study of Le Potier's Strange Duality Conjecture.  We also anticipate they will be useful in the study of ample vector bundles on these surfaces.

Let  $\F_e$ denote the Hirzebruch surface $\PP(\OO_{\PP^1} \oplus \OO_{\PP^1}(e))$, where $e$ is a nonnegative integer. The Picard group  $\Pic(\F_e) = \ZZ E \oplus \ZZ F$ is generated by the class of a fiber $F$ of the projection $\pi: \F_e \rightarrow \PP^1$ and the class of the section $E$ with self-intersection $E^2 = -e$. Let $H$ denote an ample class on $\F_e$ and let ${\bf v}$ be the Chern character of a positive rank $\mu_H$-semistable sheaf. We call ${\bf v}$ \emph{stable} for brevity. Let $M({\bf v}):=M_H^{\mu\textit{-ss}}({\bf v})$ be the moduli space of $\mu_H$-semistable sheaves with Chern character ${\bf v}$.   By a theorem of Walter \cite{Walter}, the moduli space $M({\bf v})$ is irreducible, and therefore it makes sense to talk about a general sheaf of character ${\bf v}$.  If $r({\bf v})\geq 2$, then Walter additionally shows the general sheaf in $M({\bf v})$ is a vector bundle.  

Our first theorem generalizes the G\"{o}ttsche-Hirschowitz Theorem.  

\begin{theorem}\label{thm-BNintro}
Let ${\bf v}$ be a stable Chern character on $\F_e$ with rank $r({\bf v})\geq 2$ and \emph{total slope} $$\nu({\bf v}) := \frac{c_1({\bf v})}{r({\bf v})}$$ satisfying $\nu({\bf v})\cdot F \geq -1$.  

If $\nu({\bf v}) \cdot E \geq -1$, then the general sheaf $\cV \in M({\bf v})$ has at most one nonzero cohomology group and, furthermore, $H^2(\F_e,\cV)=0$.  Conversely, if $\chi({\bf v}) \geq 0$, then the general sheaf in $M({\bf v})$ has at most one nonzero cohomology group if and only if $\nu({\bf v}) \cdot E \geq -1.$
\end{theorem}

More precisely, we will give a simple formula to compute the Betti numbers $h^i(\F_e,\cV)$ of a general sheaf $\cV\in M({\bf v})$; see Theorem \ref{thm-betti}. The statements in Theorem \ref{thm-BNintro} contain the most challenging and interesting part of this computation.  By replacing ${\bf v}$ by the Serre dual character ${\bf v}^D$, we can always reduce to the case $\nu({\bf v}) \cdot F \geq -1$, so this assumption is not really restrictive. Under the assumptions of the theorem, the Euler characteristic completely determines the cohomology of the general sheaf if $\nu({\bf v}) \cdot E \geq -1$. If $\nu({\bf v}) \cdot E < -1$ and $\chi({\bf v}) \geq 0$, then the general sheaf has both nonzero $h^0$ and $h^1$. On the other hand, if $\nu({\bf v})\cdot E <-1$ and $\chi({\bf v}) < 0$,  then general sheaf has only $h^1$ provided that the discriminant of ${\bf v}$ is sufficiently large. We will quantify this precisely in Corollary \ref{cor-BN}.  

Theorem \ref{thm-BNintro} has many applications.  For instance, it shows that effective theta divisors can be constructed on moduli spaces $M({\bf v})$ if $\chi({\bf v}) =0$, $\nu({\bf v})\cdot F \geq -1$, and $\nu({\bf v}) \cdot E \geq -1$.  In this special case, Theorem \ref{thm-BNintro} was shown in \cite{CoskunHuizengaWBN} by a different approach.  The full version of Theorem \ref{thm-BNintro} for arbitrary Euler characteristic will also play a crucial role in classifying the stable Chern characters for which the general sheaf in $M({\bf v})$ is globally generated.  The next theorem contains the majority of the classification; see Theorems \ref{thm-gg1} and \ref{thm-gg0} for the complete classification.

\begin{theorem}\label{thm-GGintro}
Let $e\geq 1$, and let ${\bf v}$ be a stable Chern character on $\F_e$. Assume that $r({\bf v}) \geq 1$, $\chi({\bf v}) \geq r({\bf v})+2$, $\nu({\bf v}) \cdot F > 0$ and $\nu({\bf v}) \cdot E \geq 0$. Then the general sheaf in $M({\bf v})$ is globally generated.
\end{theorem}

As a consequence of the $\F_1$ case of Theorem \ref{thm-GGintro}, we complete an analogous classification for $\P^2$ started by Bertram-Goller-Johnson \cite{BGJ}.   We will use two techniques to prove our theorems. We will make use of the stack of $F$-prioritary sheaves, and  we will find special resolutions of the general sheaf in the spirit of the Gaeta resolution on $\P^2$.

A torsion-free sheaf $\cV$ on $\F_e$ is \emph{$F$-prioritary} if $\Ext^2(\cV, \cV(-F))=0$. The stack of $F$-prioritary sheaves $\cP_F({\bf v})$ with Chern character ${\bf v}$ is irreducible \cite[Proposition 2]{Walter}. Furthermore, $\mu_H$-semistable sheaves are $F$-prioritary, so $\cM_H^{\mu\textit{-ss}}({\bf v})\subset \cP_F({\bf v})$ is an open substack, which is dense if it is nonempty. Hence, assuming $\mu_H$-semistable sheaves of character ${\bf v}$ exist, to show that the general $\cV\in M({\bf v})$ satisfies some open property, it suffices to exhibit one $\cV\in \cP_F({\bf v})$ with that property. The advantage of working with $F$-prioritary sheaves is that they are much easier to construct than semistable sheaves. For example, one can construct $F$-prioritary sheaves as certain direct sums of line bundles.  We will prove Theorem \ref{thm-BNintro} by explicitly constructing an $F$-prioritary sheaf with at most one nonzero cohomology group for every character ${\bf v}$ satisfying the hypotheses of the Theorem.  

\begin{remark}
In fact, $\mu_H$-semistability will play essentially no role in this paper.  Instead, in the body of the article we state almost all of our theorems for moduli stacks $\cP_F({\bf v})$ where ${\bf v}$ is a character satisfying the Bogomolov inequality $\Delta({\bf v})\geq 0$ (which is automatically satisfied if there is a $\mu_H$-semistable sheaf of character ${\bf v}$).  When the space $M({\bf v})$ is nonempty, the analogous results for $M({\bf v})$ follow immediately. 
\end{remark}

Our second technique will be to exploit a convenient resolution of the general sheaf in $M({\bf v})$. We will be able to read many cohomological properties of the general sheaf from this resolution. Our main result is as follows.  See Theorem \ref{thm-Gaeta} for a stronger statement.

\begin{theorem}\label{thm-IntroGaeta}
Suppose $e\geq 2$, and let $\cV \in M({\bf v})$ be a general $\mu_H$-semistable sheaf on $\F_e$.  Then there exists a line bundle $L$ such that $\cV$ has a resolution of the form
$$0 \to L(-E-(e+1)F)^{\alpha} \to L(-E-eF)^{\beta} \oplus L(-F)^{\gamma} \oplus L^{\delta} \to \cV \to 0.$$
\end{theorem}

We use Theorem \ref{thm-IntroGaeta} in order to analyze the most challenging case in the classification of Chern characters such that the general bundle is globally generated.  We additionally anticipate resolutions of this type will be useful in studying various questions related to generic vector bundles on Hirzebruch surfaces.

\subsection*{Organization of the paper} In \S \ref{sec-Prelim}, we collect preliminary facts concerning the geometry of Hirzebruch surfaces and moduli of sheaves.  In \S \ref{sec-BN}, we prove a strengthened version of Theorem \ref{thm-BNintro} and compute the Betti numbers of a general sheaf. In \S \ref{sec-Gaeta}, we prove a strengthened version of Theorem \ref{thm-IntroGaeta}.  We then classify characters such that the general sheaf is globally generated in \S\ref{sec-gg}.  We close the paper with some remarks on the open question of determining the Chern characters of ample vector bundles.

\subsection*{Acknowledgements} We would like to thank Aaron Bertram, Lawrence Ein, John Lesieutre, and Daniel Levine for useful discussions. 

\section{Preliminaries}\label{sec-Prelim}

In this section, we collect basic facts which we will use in the rest of the paper.

\subsection{Hirzebruch surfaces} We refer the reader to \cite{Beauville},  \cite{CoskunScroll} and \cite{Hartshorne} for detailed expositions on Hirzebruch surfaces.  Let $e$ be a nonnegative integer and let $\F_e$ denote the ruled surface $\PP(\OO_{\PP^1} \oplus \OO_{\PP^1}(e))$. Let $\pi: \F_e \rightarrow \PP^1$ be the natural projection. Let $F$ be the class of a fiber and let $E$ be the class of the section of self-intersection $-e$. The intersection pairing on $\F_e$ is given by $$E^2 = -e, \quad F^2 =0, \quad F \cdot E = 1.$$ The effective cone of $\F_e$ is generated by $E$ and $F$. In fact, $$h^0(\F_e,\OO_{\F_e}( aE + bF)) > 0$$ if and only if $a, b \geq 0$. 
Dually, the nef cone of $\F_e$ is generated by $E+eF$ and $F$. The canonical divisor is given by $K_{\F_e} = -2E - (e+2)F$. By Serre duality, $h^2(\F_e, \OO_{\F_e}(aE + bF)) > 0$ if and only if $a \leq -2$ and $b \leq -e-2$. The following theorem summarizes the cohomology of line bundles on $\F_e$ (also see \cite{CoskunScroll, CoskunHuizengaWBN, Hartshorne}).

\begin{theorem}\label{thm-lineBundle}
Let $L=\OO_{\F_e}(aE + bF)$ be a line bundle on $\F_e$.  Then \begin{enumerate}
\item We have $$\chi(L) = (a+1)(b+1)-e \frac{a(a+1)}2.$$
\item If $L\cdot F \geq -1$, then $h^2(\F_e,L)=0$.  
\item If $L\cdot F \leq -1$, then $h^0(\F_e,L)=0$.
\item In particular, if $L\cdot F = -1$, then $L$ has no cohomology in any degree.
\end{enumerate} Now suppose $L\cdot F > -1$.  Then $h^2(\F_e,L)=0$, so either of the numbers $h^0(\F_e,L)$ or $h^1(\F_e,L)$ determine the cohomology of $L$.  These can be determined as follows.
\begin{enumerate}

\item[(5)] If $L\cdot E \geq -1$, then $H^1(\F_e,L)=0$, and so $h^0(\F_e,L) = \chi(L)$. 

\item[(6)] If $L\cdot E < -1$, then $H^0(\F_e,L) \cong H^0(\F_e,L(-E))$, and so the cohomology of $L$ can be determined inductively using (3) and (5).
\end{enumerate}
(If $L\cdot F < -1$ then the cohomology of $L$ can be determined by Serre duality.)
\end{theorem}

While the proof is well-known, we include it since the argument is relevant to our approach for vector bundles.

\begin{proof}
Part (1) is just Riemann-Roch.  Part (3) comes from the description of the effective cone, and (2) follows by Serre duality.  Part (4) is just a combination of (1)-(3).

(5) Suppose $a=L\cdot F \geq 0$ and $L\cdot E \geq -1$, and consider the restriction sequence $$0\to L(-E)\to L \to L|_E\to 0.$$  Then $H^1(E,L|_E)=0$, so $H^1(\F_e,L)$ is a quotient of $H^1(\F_e,L(-E))$.  Repeating this process, we eventually find that $H^1(\F_e,L)$ is a quotient of $H^1(\F_e,L(-(a+1)E)) = H^1(\F_e,\OO_{\F_e}(-E+bF))=0$ by (4).  Therefore $H^1(\F_e,L)=0$.

(6) The isomorphism $H^0(\F_e,L)\cong H^0(\F_e,L(-E))$ comes immediately from the restriction sequence.  When we twist $L$ by $-E$, the intersection number with $E$ increases by $e$ and the intersection number with $F$ decreases by $1$.  Thus there is some smallest integer $m > 0$ such that either $L(-mE)\cdot F \leq -1$ or $L(-mE) \cdot E \geq -1$, and by induction $h^0(\F_e,L) = h^0(\F_e,L(-mE))$.  If $L(-mE)\cdot F \leq -1$, we have $h^0(\F_e,L)= 0$ by (3).  On the other hand, if $L(-mE)\cdot F > -1$ and $L(-mE)\cdot E \geq -1$, then we have $h^0(\F_e,L) = \chi(L(-mE))$ by (5).  
\end{proof}

\begin{remark} 
In particular, the line bundles $$\OO_{\F_e}(-F), \quad \OO_{\F_e}(-2E-(e+1)F), \quad {\mbox{and}} \quad \OO_{\F_e}(-E+bF) \quad (b\in \Z)$$ all have no cohomology in any degree.
\end{remark}

\subsection{Numerical invariants and semistability} We refer the reader to \cite{CoskunHuizengaGokova, HuybrechtsLehn, LePotier} for more details on moduli spaces of vector bundles on surfaces.  Let $X$ be a surface and let $H$ be an ample divisor on $X$.  For a sheaf $\cV$ (or Chern character ${\bf v}$) of positive rank we respectively define the \emph{$H$-slope}, \emph{total slope}, and \emph{discriminant:} $$\mu_H(\cV) = \frac{c_1(\cV)\cdot H}{r({\cV})H^2} \qquad \nu({\bf v}) = \frac{c_1(\cV)}{r(\cV)} \qquad \Delta({\cV}) = \frac{1}{2}\nu({\cV})^2 - \frac{\ch_2({\cV})}{r(\cV)}.$$ The discriminant has the following important properties: \begin{enumerate} \item If $L$ is a line bundle then $\Delta(L) = 0$.  
\item If $\cV$ is torsion-free and $\cW$ is a nonzero vector bundle, then $\Delta(\cV \te \cW) = \Delta(\cV) + \Delta(\cW)$.  In particular, $\Delta(\cV\te L) = \Delta(\cV)$ for any line bundle $L$.
\end{enumerate}
If we put $P(\nu) = \chi(\OO_X)+\frac{1}{2}\nu\cdot (\nu-K_X)$, then the Riemann-Roch formula reads $$\chi(\cV) = r(\cV)(P(\nu(\cV))-\Delta(\cV)).$$ In the special case of $\F_e$, if $\cV$ has rank $r$ and total slope $\nu(\cV) = \frac{k}{r} E + \frac{l}{r} F$, then the term $P(\nu(\cV))$ becomes $$P(\nu(\cV)) = \left(\frac{k}{r}+1\right)\left(\frac{l}{r}+1\right) - \frac{ek}{2r}\left(\frac{k}{r}+1\right).$$ 
We will find the following easy consequence of Hirzebruch-Riemann-Roch useful.
\begin{lemma}\label{lem-eulerchar}
Let $X$ be a smooth projective surface with canonical class $K_X$. Let $\cV$ be a torsion-free sheaf of rank $r$ on $X$ and let $L$ be a line bundle on $X$. Then 
$$\chi(\cV \otimes L) = \chi(\cV) + L \cdot c_1 (\cV) + r(\chi(L)-\chi(\OO_X)).$$
\end{lemma}

A torsion-free sheaf $\cV$ on $X$ is \emph{$\mu_H$-semistable} if whenever $\cW\subset \cV$ is a nonzero subsheaf we have $\mu_H(\cW) \leq \mu_H(\cV)$.  Ordinary $H$-Gieseker semistability implies $\mu_H$-semistability, but all the theorems in this paper will hold for the weaker $\mu_H$-semistability.  The Bogomolov inequality implies that $\Delta(\cV)\geq 0$ for any $\mu_H$-semistable sheaf.

\subsection{Prioritary sheaves}  While our results are perhaps the most interesting in the context of semistable sheaves, stronger results can be proved by working with the easier notion of prioritary sheaves.

\begin{definition}
Let $D$ be a divisor on a smooth surface $X$. A torsion free sheaf $\cV$ is called \emph{$D$-prioritary}  if $\Ext^2(\cV, \cV(-D))=0$.

We write $\cP_D({\bf v})$ (or $\cP_{X,D}({\bf v})$ if $X$ is not clear) for the stack of $D$-prioritary sheaves on $X$.
\end{definition}

In this paper, we will primarily consider $F$-prioritary sheaves on $\F_e$, where $F$ is the class of a fiber. If $H$ is any ample class on $\F_e$ and $\cV$ is a torsion-free $\mu_H$-semistable sheaf, then $\cV$ is automatically $F$-prioritary. Indeed, suppose $\cV$ is $\mu_H$-semistable.  By Serre duality
$$\Ext^2(\cV, \cV(-F))= \Hom(\cV, \cV(K_{\F_e}+F))^*.$$ Since $K_{\F_e}+F$ is anti-effective, we have $(K_{\F_e}+F)\cdot H < 0$.  Therefore $\mu_H(\cV) > \mu_H(\cV(K_{\F_e}+F))$, and $\Hom(\cV,\cV(K_{\F_e}+F))=0$ by $\mu_H$-semistability.  Therefore $\cV$ is $F$-prioritary.  It follows from openness of stability that the stack $\cM_H^{\mu{\textit{-ss}}}({\bf v})$ is an open substack of $\cP_F({\bf v})$.  Furthermore, if there are $\mu_H$-semistable sheaves of character ${\bf v}$, then it is dense by the following theorem of Walter \cite{Walter}.

\begin{theorem}\label{thm-Walter}
Let $\pi:X\to \P^1$ be a geometrically ruled surface with fiber class $F$, and let ${\bf v}\in K(X)$ have positive rank.  Then the stack $\cP_F({\bf v})$ is irreducible. Furthermore, if $r({\bf v}) \geq 2$ and $\cP_F({\bf v})$ is nonempty, then a general $\cV\in \cP_F({\bf v})$ is a vector bundle.
\end{theorem}

In particular, to show that a general $\mu_H$-semistable sheaf on $\F_e$ of character ${\bf v}$ satisfies some open property, it is sufficient to produce an $F$-prioritary sheaf with that property.  

Every vector bundle on a rational curve is a direct sum of line bundles $\oplus_{i=1}^r\OO_{\PP^1}(a_i)$. The vector bundle is {\em balanced} if $|a_i -a_j| \leq 1$ for all $i, j$. The next proposition explains the importance of the prioritary condition.

\begin{proposition}\label{prop-balanced}
Let $C$ be a curve on a surface $X$ and let $\cF_s/S$ be a complete family of $C$-prioritary sheaves which are locally free on $C$.  Then the restricted family $\cF_s|_C/S$ is complete. 

Therefore, if $C$ is a rational curve, then $\cF_s|_C$ is balanced for all $s$ in an open dense subset of $S$.
\end{proposition}

\begin{proof}
The Kodaira-Spencer map of the restricted family $$T_s S \rightarrow \Ext^1(\cF_s|_C, \cF_s|_C)$$ factors as a composition 
$$T_sS \rightarrow  \Ext^1(\cF_s,\cF_s) \rightarrow  \Ext^1_C (\cF_s|_C,\cF_s|_C).$$ The first map is surjective since the family $\cF$ is a complete family. We have an identification
\begin{align*}\Ext^1(\cF_s, \cF_s|_C) &= H^1(X, \cF_s^*  \otimes \cF_s|_C) \\&= H^1(C, \cF_s^*|_C \otimes \cF_s|_C) \\&= \Ext^1_C (\cF_s|_C, \cF_s|_C)\end{align*}
since $\cF_s$ is locally free along C. Hence,  the second map in the factorization appears in the long exact sequence obtained by applying $\Hom(\cF_s, \cF_s \otimes -)$ to
$$0 \rightarrow  \OO_X(-C) \rightarrow  \OO_X  \rightarrow \OO_C \rightarrow 0.$$
By the long exact sequence of cohomology 
$$\Ext^1(\cF_s,\cF_s) \rightarrow \Ext^1(\cF_s, \cF_s|_C) \rightarrow  \Ext^2(\cF_s, \cF_s(-C)) =0$$
we conclude that the Kodaira-Spencer map is surjective, and $\cF_s|_C/S$ is complete.  

In any complete family of vector bundles on $\P^1$ parameterized by an irreducible base, the general bundle is balanced.  The second statement follows.
\end{proof}

One advantage of working with prioritary sheaves is that they are well-behaved under elementary modifications.  The following result is well-known but we include the proof for completeness and lack of a single streamlined reference.

\begin{lemma}\label{lem-elementaryprioritary}
Let $L$ be a line bundle on a smooth surface $X$.  Let $\cV$ be a torsion-free sheaf on $X$, and let $\cV'$ be a general elementary modification of $\cV$ at a general point $p\in X$, defined as the kernel of a general surjection $\phi:\cV\to \OO_p$: $$0\to \cV'\to \cV \stackrel{\phi}{\to} \OO_p\to 0.$$
\begin{enumerate}
\item\label{elem1} If $\cV$ is $L$-prioritary, then $\cV'$ is $L$-prioritary.
\item\label{elem2} The sheaves $\cV$ and $\cV'$ have the same rank and $c_1$, and \begin{align*}\chi(\cV') &= \chi(\cV)-1\\
\Delta(\cV') &= \Delta(\cV) + \frac{1}{r}.\end{align*}
\item\label{elem3} We have $H^2(X,\cV) \cong H^2(X,\cV')$.
\item\label{elem4} If at least one of $H^0(X,\cV)$ or $H^1(X,\cV)$ is zero, then at least one of $H^0(X,\cV')$ or $H^1(X,\cV')$ is zero.  In particular, if $H^2(X,\cV)=0$ and $\cV$ has at most one nonzero cohomology group, then $H^2(X,\cV') = 0$ and $\cV'$ has at most one nonzero cohomology group.
\end{enumerate}
\end{lemma}
\begin{proof}
(\ref{elem1}) Clearly $\cV'$ is torsion-free since $\cV$ is.  We have $\Ext^2(\cV, \cV (-L))=0$ since $\cV$ is $L$-prioritary. We would like to show that $\Ext^2(\cV', \cV'(-L))=0$.  Applying $\Ext(-, \cV'(-L))$ to the sequence, we obtain a surjection $$\Ext^2( \cV, \cV'(-L)) \rightarrow \Ext^2(\cV', \cV'(-L)) \to  0.$$ Hence, it suffices to show that $\Ext^2( \cV, \cV'(-L))=0$. Applying $\Ext(\cV, -)$ to the sequence twisted by $-L$,  we obtain
$$\Ext^1(\cV, \OO_p) \rightarrow \Ext^2(\cV, \cV' (-L)) \rightarrow \Ext^2(\cV, \cV (-L)).$$ We have $\Ext^2(\cV, \cV (-L))=0$ by the assumption that $\cV$ is $L$-prioritary and  $\Ext^1(\cV, \OO_p)=0$ because $\cV$ is locally free at the general point $p$. We conclude that $\cV'$ is $L$-prioritary.

(\ref{elem2}) The  first equality follows from the exact sequence, and Riemann-Roch gives the second.

(\ref{elem3}) This follows immediately from the long exact sequence in cohomology.

(\ref{elem4}) Consider the long exact sequence in cohomology $$0\to H^0(X,\cV')\to H^0(X,\cV) \to H^0(X,\OO_p)\to H^1(X,\cV')\to H^1(X,\cV) \to 0.$$ If $H^0(X,\cV) = 0$, then $H^0(X,\cV') = 0$.  Suppose $H^0(X,\cV)\neq 0$ and $H^1(X,\cV) = 0$.  Then by the choice of $\phi:\cV\to \OO_p$, the map $H^0(X,\cV)\to H^0(X,\OO_p) = \CC$ is surjective.  It follows that $H^1(X,\cV')= 0$.
\end{proof}

\subsection{Exceptional collections} We refer the reader to \cite{CoskunHuizengaWBN} for more details on the following discussion.  We will use exceptional collections of sheaves on $\F_e$ as basic building blocks for constructing sheaves with useful properties.

\begin{definition}
A coherent sheaf $\cE$ is called  {\em exceptional} if $\Ext^i(\cE, \cE) =0$ for $i \geq 1$ and $\Hom(\cE, \cE) =\CC$.  An ordered collection of exceptional objects $\cE_1, \dots, \cE_m$ is an {\em exceptional collection}  if $\Ext^i(\cE_t, \cE_s)=0$ for $s< t$ and all $i$. The exceptional collection is {\em strong}  if in addition $\Ext^i(\cE_s, \cE_t)=0$ for $s< t$ and $i\not= 0$. 
\end{definition}

\begin{example}
On $\F_e$, the line bundles $$\OO_{\F_e}(-E-(e+1)F), \OO_{\F_e}(-E-eF), \OO_{\F_e}(-F), \OO_{\F_e}$$ give a strong exceptional collection (see \cite[Example 3.2]{CoskunHuizengaWBN}).
\end{example}

In \cite{CoskunHuizengaWBN}, the majority of the following result was proved.

\begin{theorem}\label{thm-resComplete}
Let $\cE_1, \ldots, \cE_m, \cF_1,\ldots,\cF_n$ be a strong exceptional collection of vector bundles on a surface $X$, partitioned into two blocks, and let $L$ be a line bundle on $X$.  Assume that 
\begin{enumerate}
\item The sheaf $\sHom(\cE_i, \cF_j)$ is globally generated for all $i,j$,
\item $\Ext^1(\cE_i, \cF_j(-L))=0$ for all $i,j$, and
\item $\Ext^2(\cF_i, \cF_j(-L))=0$ for all $i,j$.
\end{enumerate}
Suppose $a_1,\ldots,a_m$ and $b_1,\ldots,b_n$ are nonnegative integers such that $\sum b_j r(\cF_j) - \sum a_i r(\cE_i) > 0$, and let $$U \subset \Hom\left(\bigoplus_{i=1}^m \cE_i^{a_i}, \bigoplus_{j=1}^n \cF_j^{b_j}\right)$$ be the open subset parameterizing injective sheaf maps with torsion-free cokernel.  For $\phi\in U$, let $\cV_\phi$ be the cokernel:  $$0\to \bigoplus_{i=1}^m \cE_i^{a_i} \fto{\phi} \bigoplus_{j=1}^n \cF_j^{b_j} \to \cV_\phi \to 0.$$
Then $U$ is nonempty, and the family $\cV_\phi/U$ is a complete family of $L$-prioritary sheaves.
\end{theorem}
\begin{proof}
If $r(\cV_\phi) = \sum b_j r(\cF_j) - \sum a_i r(\cE_i) \geq 2$, then this is  \cite[Proposition 3.6]{CoskunHuizengaWBN}.    When $r(\cV_\phi) = 1$, everything follows as in the $r(\cV_\phi) \geq 2$ case, except we must verify that $\cV_\phi$ is still torsion-free.  The rank of $\phi$ only drops in codimension $2$, so any torsion subsheaf $\cT\subset \cV_\phi$ has $0$-dimensional support.  If $\cT\neq 0$ and $M$ is any line bundle, then $H^0(X,\cT\te M) \neq 0$ and hence $H^0(X,\cV_\phi \te M)\neq 0$.  However, if $M$ is sufficiently anti-ample, then tensoring the resolution of $\cV_\phi$ by $M$ shows $H^0(X,\cV_\phi\te M) = 0$.  Therefore $\cT = 0$.
\end{proof}

\subsection{Globally generated vector bundles} We now recall some properties of globally generated sheaves for use in \S\ref{sec-gg}.  A coherent sheaf $\cV$ on a projective variety $X$ over a field $k$ is {\em globally generated} if the evaluation map $$H^0(X, \cV) \otimes_k \OO_X \to \cV$$ is surjective. The following lemma is immediate.

\begin{lemma}\label{lem-quotientofgg}
Any quotient of a globally generated sheaf is globally generated. 
\end{lemma}

If $\cV$ is globally generated, then there is an obvious restriction on $c_1(\cV)$.

\begin{lemma}\label{lem-ggnef}
Let $\cV$ be a globally generated vector bundle on $X$. Then for every  curve $C \subset X$, we have $$C \cdot c_1(\cV) \geq 0.$$
\end{lemma}

\begin{proof}
The  standard short exact sequence $$0 \rightarrow \cV(-C) \rightarrow \cV \rightarrow \cV|_C \rightarrow 0$$ implies that $\cV|_C$ is a quotient of $\cV$. Hence, by Lemma \ref{lem-quotientofgg}, $\cV|_C$ is globally generated. Since a globally generated vector bundle on a curve has nonnegative degree, we conclude that $C \cdot c_1(\cV) \geq 0.$
\end{proof}

We warn that while many properties of sheaves that we consider in this paper are open properties, the property of being globally generated is not in general open.  However, in a family $\cV/S$ of sheaves with no higher cohomology, the locus of globally generated sheaves is clearly open.  The next example shows that in an arbitrary family the locus of globally generated sheaves can fail to be open.

\begin{example}  
On $\P^2$, the locus in $M({\bf v})$ of globally generated semistable sheaves is not generally open, and even if it is nonempty, it need not be dense.   For example, for any $d\geq 1$ the bundle $\cV$ $$0\to \OO_{\P^2}(-d)\to \Hom(\OO_{\P^2}(-d),\OO_{\P^2})^* \te \OO_{\P^2} \to \cV\to 0$$ given by the cokernel of the coevaluation map is globally generated and semistable \cite[Lemma 9.2.3]{LePotier}.  Let ${\bf v} = \ch \cV$.  By the Brill-Noether theorem of G\"ottsche and Hirschowitz \cite{GottscheHirschowitz}, a general sheaf in $M({\bf v})$ has only one nonzero cohomology group.
\begin{itemize}
\item If $d=1$ or $2$, then the bundle $\cV$ is exceptional, and $M({\bf v})=\{[\cV]\}$ is a point.
\item If $d=3$, then $\chi({\bf v}) = r({\bf v})$, the moduli space $M({\bf v})$ is positive-dimensional, and the general sheaf in $M({\bf v})$ has $r({\bf v})$ sections but is not globally generated.
\item If $d\geq 4$, then $\chi({\bf v}) = 3d < {d+2 \choose 2}-1 = r({\bf v})$ and the general sheaf doesn't even have $r({\bf v})$ sections, so has no hope of being globally generated.
\end{itemize}
Thus, sheaves with ``extra'' sections can be globally generated even if the general sheaf is not.
\end{example}

\section{The Brill-Noether theorem}\label{sec-BN}
In this section, we prove the  Brill-Noether theorem for vector bundles on Hirzebruch surfaces and compute the Betti numbers of a general sheaf.  The next result summarizes the various results in this section; it is instructive to compare the statement with the line bundle case, Theorem \ref{thm-lineBundle}.

\begin{theorem}\label{thm-betti}
Let ${\bf v}\in K(\F_e)$ be a Chern character with positive rank $r=r({\bf v})$ and $\Delta({\bf v}) \geq 0$.  Then the stack $\cP_F({\bf v})$ is nonempty and irreducible. Let $\cV\in \cP_F({\bf v})$ be a general sheaf.
\begin{enumerate}
\item If we write $\nu({\bf v}) = \frac{k}{r}E + \frac{l}{r}F$, then \begin{align*}\chi({\bf v}) & = r(P(\nu({\bf v}))-\Delta({\bf v})) \\&= r \left(\left(\frac{k}{r}+1\right)\left(\frac{l}{r}+1\right) - \frac{ek}{2r}\left(\frac{k}{r}+1\right)-\Delta({\bf v})\right)\end{align*}
\item If $\nu({\bf v})\cdot F \geq -1$, then $h^2(\F_e,\cV)=0$.
\item If $\nu({\bf v})\cdot F \leq -1$, then $h^0(\F_e,\cV)=0$.
\item In particular, if $\nu({\bf v})\cdot F = -1$, then $h^1(\F_e,\cV) = - \chi({\bf v})$ and all other cohomology vanishes.
\end{enumerate}
Now suppose $\nu({\bf v})\cdot F > -1$.  Then $H^2(\F_e,\cV)=0$, so either of the numbers $h^0(\F_e,\cV)$ or $h^1(\F_e,\cV)$ determine the Betti numbers of $\cV$.  These can be determined as follows.
\begin{enumerate}
\item[(5)] If $\nu({\bf v})\cdot E \geq -1$, then $\cV$ has at most one nonzero cohomology group.  Thus if $\chi({\bf v})\geq 0$, then $h^0(\F_e,\cV) = \chi({\bf v})$, and if $\chi({\bf v})\leq 0$, then $h^1(\F_e,\cV) = -\chi({\bf v})$.
\item[(6)] If $\nu({\bf v})\cdot E < -1$, then $H^0(\F_e,\cV) \cong H^0(\F_e,\cV(-E))$, and so the Betti numbers  of $\cV$ can be determined inductively using (3) and (5).
\end{enumerate}
(If $\nu({\bf v})\cdot F < -1$ and $r({\bf v})\geq 2$, then the cohomology of $\cV$ can be determined by Serre duality.)
\end{theorem}

\begin{remark}
In particular, if $H$ is an arbitrary ample divisor on $\F_e$ and there are $\mu_H$-semistable sheaves of character ${\bf v}$, then Theorem \ref{thm-betti} allows us to compute the Betti numbers of a general sheaf $\cV\in M_H^{\mu\textit{-ss}}({\bf v})$.
\end{remark}

Statement (1) of the theorem is just Riemann-Roch, reproduced here for the reader's convenience.  The locus of sheaves $\cV\in \cP_F({\bf v})$ satisfying each of the statements (2)-(5) of the theorem is open in the stack $\cP_F({\bf v})$, so it suffices to produce a single sheaf $\cV\in \cP_F({\bf v})$ with the given cohomology.  We will construct such sheaves by using direct sums of line bundles as basic building blocks (\S \ref{ssec-ds}) and applying elementary modifications to them (\S \ref{ssec-elementary}).  Statement (6) follows easily from the observation that for a general $\cV\in \cP_F({\bf v})$ the restriction $\cV|_E$ is balanced (see Corollary \ref{cor-balanced}).

\subsection{Prioritary direct sums of line bundles}\label{ssec-ds}  We first show that for any rank $r\geq 1$ and slope $\nu = \frac{k}{r}E + \frac{l}{r}F$, there is a prioritary direct sum of line bundles with $\Delta \leq 0$. We identify $N^1(\F_e)_\Q \cong \Q^2$ with the $(\frac{k}{r},\frac{l}{r})$-plane.  The particular line bundles we use depend on $(\frac{k}{r},\frac{l}{r})$.  The next lemma shows that if this point is in a certain triangular region, then we can find such a direct sum of line bundles.

\begin{lemma}\label{lem-ds}
Let $a,b,c\geq 0$ be nonnegative integers, and let $$\cW = \OO_{\F_e}(-E-(e+1)F)^a \oplus \OO_{\F_e}(-F)^b \oplus \OO_{\F_E}^c.$$
\begin{enumerate}
\item The bundle $\cW$ is $F$-prioritary and $E$-prioritary, and $\cW$ has no higher cohomology.
\item We have $\Delta(\cW)\leq 0$.
\item Let $r\geq 1$ and represent a total slope $\frac{k}{r}E + \frac{l}{r}F \in N^1(\F_e)_\Q$ by the point $(\frac{k}{r},\frac{l}{r})\in \Q^2$.  
If $(\frac{k}{r},\frac{l}{r})$ is in the convex region with vertices $$(-1,-e-1),\, (0,-1),\, (0,0),$$ then it is the slope of a rank $r$ bundle $\cW$ as above.

\end{enumerate}
\end{lemma}
\begin{proof}
(1) The vector space $\Ext^2(\cW, \cW(-F))$ is a direct sum of vector spaces of the form $H^2(\F_e, \OO_{\F_e}(\alpha E +\beta F))$ where $\alpha \geq -1$, and these are zero.  Therefore $\cW$ is $F$-prioritary.  Similarly, the vector space $\Ext^2(\cW,\cW(-E))$ is a direct sum of vector spaces of the form $H^2(\F_e,\OO_{\F_e}(\alpha E+\beta F))$ where $\beta \geq -e-1$, and these are again zero.  Therefore $\cW$ is $E$-prioritary.  The statement on cohomology follows at once from Theorem \ref{thm-lineBundle}.

(2) We compute $$\ch(\cW) = \left(a+b+c,-aE-(a(e+1)+b)F,\frac{a(e+2))}{2}\right).$$ Then \begin{align*}2r(\cW)^2\Delta(\cW) &=  \ch_1(\cW)^2-2r(\cW)\ch_2(\cW)\\&=-a^2e+2a(a(e+1)+b)-(a+b+c)a(e+2)\\
&= -a(be+ce+2c).\end{align*}
Therefore $\Delta(\cW)\leq  0$.

(3) The slope $(\frac{k}{r},\frac{l}{r})$ is in the convex region spanned by $(-1,-e-1)$, $(0,-1)$ and $(0,0)$ if and only if the vector $(k,l,r)$ is in the cone spanned by $(-1,-e-1,1)$, $(0,-1,1)$, and $(0,0,1)$.  This happens if and only if the linear system $$\begin{pmatrix}-1 & 0 & 0 \\ -e-1 & -1 & 0\\ 1 & 1 & 1\end{pmatrix}\begin{pmatrix} a\\ b \\ c\end{pmatrix} = \begin{pmatrix} k\\ l \\ r \end{pmatrix}$$ has a solution $(a,b,c) \in \Q^3_{\geq 0}$. But the matrix is in $\SL_3(\Z)$, so actually $(a,b,c)\in \Z^3_{\geq 0}$.  The corresponding bundle $\cW$ has rank $r$ and slope $(\frac{k}{r},\frac{l}{r})$.
\end{proof}

By an analogous construction we can further handle slopes lying in a larger quadrilateral region.  This quadrilateral has the advantage that its shifts by line bundles tile the whole $(\frac{k}{r},\frac{l}{r})$-plane.

\begin{corollary}\label{cor-quad}
Suppose $r\geq 1$ and the point $(\frac{k}{r},\frac{l}{r})$ lies in the parallelogram $Q$ bounded by the four vertices $$(-1,-e), \quad (0,0), \quad (0,-1) \quad \mbox{and}\quad (-1,-e-1).$$
Then there is a rank $r$ direct sum $\cW$ of copies of the line bundles $$\OO_{\F_e}(-E-eF),\quad  \OO_{\F_e},\quad  \OO_{\F_e}(-F),\quad \mbox{and} \quad \OO_{\F_e}(-E-(e+1)F)$$ such that  \begin{enumerate}
\item $\nu(\cW) = \frac{k}{r}E+ \frac{l}{r}F$,
\item $\cW$ is $F$- and $E$-prioritary, and
\item $\Delta(\cW) \leq 0$.
\end{enumerate}
\end{corollary}
\begin{proof}
The quadrilateral region is split into a lower triangle $T_1$ and an upper triangle $T_2$ by the line segment from $(-1,-e-1)$ to $(0,0)$.  If $(\frac{k}{r},\frac{l}{r})$ is in $T_1$, then everything follows from Lemma \ref{lem-ds}.

On the other hand, if $(\frac{k}{r},\frac{l}{r})$ lies in $T_2$, consider direct sums of the form $$\cW = \OO_{\F_e}^a\oplus \OO_{\F_e}(-E-eF)^b\oplus  \OO_{\F_e}(-E-(e+1)F)^c.$$  Notice that $$(\cW(E+(e+1)F)))^* = \OO_{\F_e}(-E-(e+1)F)^a \oplus \OO_{\F_e}(-F)^b \oplus \OO_{F_e}^c$$ is of the same form as the line bundles considered in Lemma \ref{lem-ds}.  Since tensoring by line bundles and taking duals preserves discriminants and prioritariness, it follows that the integers $a,b,c$ can be chosen so that $\cW$ has the required properties.
\end{proof}

\subsection{Elementary modifications}\label{ssec-elementary} When combined with elementary modifications, Corollary \ref{cor-quad} has many consequences, which we now investigate.  The proof of the next corollary is fundamental to all the results which follow and makes crucial use of  Lemma \ref{lem-elementaryprioritary}.

\begin{corollary}\label{cor-exist}
Let ${\bf v}\in K(\F_e)$ be a  character of positive rank with $\Delta({\bf v})\geq 0$.  Then $\cP_F({\bf v})$ is nonempty.
\end{corollary}
\begin{proof}
Since integer translates of the region $Q$ in Corollary \ref{cor-quad} tile $N^1(\F_e)_\Q$, we may find a line bundle $L$ such that $\nu({\bf v}(-L))$ lies in $Q$.  Then Corollary \ref{cor-quad} produces an $F$-prioritary sheaf $\cW$ of nonpositive discriminant with $r(\cW(L)) = r({\bf v})$ and $\nu(\cW(L)) = \nu({\bf v})$.  Then since $\Delta({\bf v}) \geq 0$, the integer $m = \chi(\cW(L))-\chi({\bf v})$ is nonnegative, since $$m=\chi(\cW(L)) - \chi({\bf v}) = r({\bf v}) (\Delta({\bf v})-\Delta(\cW(L))).$$ Thus by Lemma \ref{lem-elementaryprioritary}, if we perform $m$ general elementary modifications to $\cW(L)$, the resulting sheaf $\cV$ is $F$-prioritary and has $\ch \cV = {\bf v}$.
\end{proof}

By essentially the same proof, we can deduce that a general $\cV\in \cP_F({\bf v})$ splits as a balanced direct sum on the exceptional section $E$.

\begin{corollary}\label{cor-balanced}
Let ${\bf v}\in K(\F_e)$ be a  character of positive rank with $\Delta({\bf v})\geq 0$.  Then the general $\cV\in \cP_F({\bf v})$ is $E$-prioritary, and furthermore $\cV|_E$ is a balanced direct sum of line bundles.
\end{corollary}
\begin{proof}
In the proof of Corollary \ref{cor-exist}, the bundle $\cW$ is $E$-prioritary, and hence so is $\cW(L)$.  Then the sheaf $\cV$ is also $E$-prioritary by Lemma \ref{lem-elementaryprioritary}. Furthermore, $\cW|_E$ is clearly balanced, and hence so is $\cW(L)|_E$ and  $\cV|_E$.  (Alternately, the balancedness of $\cV|_E$ follows from the general result Proposition \ref{prop-balanced}.)
\end{proof}

When further information about ${\bf v}$ is known, the possibilities for the twisting line bundle $L$ such that ${\bf v}(-L)$ lies in $Q$ are restricted.  This fact allows us to deduce further results about sheaves of character ${\bf v}$.

\begin{corollary}\label{cor-betti}
Let ${\bf v}\in K(\F_e)$ be a character of positive rank with $\Delta({\bf v}) \geq 0$.  Let $\cV\in \cP_F({\bf v})$ be general.
\begin{enumerate}
\item If $\nu({\bf v})\cdot F \geq -1$, then $h^2(\F_e,\cV)=0$.
\item If $\nu({\bf v})\cdot F \leq -1$, then $h^0(\F_e,\cV)=0$.
\item In particular, if $\nu({\bf v})\cdot F = -1$, then $h^1(\F_e,\cV) = - \chi({\bf v})$ and all other cohomology vanishes.
\item Suppose $\nu({\bf v})\cdot F > -1$ and $\nu({\bf v})\cdot E \geq -1$.  Then $\cV$ has at most one nonzero cohomology group.
 \end{enumerate}
\end{corollary}
\begin{proof}
(1) In the notation of the proof of Corollary \ref{cor-exist}, since $\nu({\bf v})\cdot F \geq -1$ we may assume the line bundle $L$ is of the form $\OO_{\F_e}(aE+bF)$ with $a\geq 0$ and $b\in \Z$.  Then the line bundles appearing in $\cW(L)$ (which are twists by $L$ of the line bundles in the statement of Corollary \ref{cor-quad}) all clearly have no $h^2$.  Then $\cV$ also has no $h^2$ by Lemma \ref{lem-elementaryprioritary}.

(2) In this case, we may assume the line bundle $L$ is of the form $\OO_{\F_e}(aE+bF)$ with $a\leq -1$ and $b\in \Z$.  The line bundles appearing in $\cW(L)$ all have no $h^0$.  Then $\cV$ is a subsheaf of $\cW(L)$, so also has no $h^0$. 

(3) This immediately follows from (1) and (2).

(4) This time we may assume the line bundle $L$ is nef.  Then by Theorem \ref{thm-lineBundle}, the line bundles appearing in $\cW(L)$ all have no higher cohomology.  By Lemma \ref{lem-elementaryprioritary}, performing general elementary modifications on $\cW(L)$ results in a bundle $\cV$ with at most one nonzero cohomology group.
\end{proof}

We can combine the last two corollaries to compute the Betti numbers of a general sheaf in the remaining cases.  Together with the other results in this subsection, this completes the proof of Theorem \ref{thm-betti}.

\begin{proposition}\label{prop-splitE}
Let ${\bf v}\in K(\F_e)$ be a character of positive rank with $\Delta({\bf v})\geq 0$, $\nu({\bf v})\cdot F >-1$ and $\nu({\bf v}) \cdot E < -1$. Let $\cV\in \cP_F({\bf v})$ be general.  Then $H^0(\F_e,\cV)\cong H^0(\F_e,\cV(-E))$, and the Betti numbers of $\cV$ can be determined inductively.
\end{proposition}
\begin{proof}
Since $\cV|_E$ is balanced by Corollary \ref{cor-balanced} and $\nu({\bf v})\cdot E <-1$, we conclude that $H^0(E,\cV|_E)=0$.  Then the restriction sequence $$0\to \cV(-E)\to \cV\to \cV|_E\to 0$$ implies $H^0(\F_e,\cV)\cong H^0(\F_e,\cV(-E))$.

It remains to explain how to compute the Betti numbers of $\cV$ inductively. Twisting $\cV$ by $-E$ has the effect of increasing $\nu({\bf v})\cdot E$ by $e$ and decreasing $\nu({\bf v})\cdot F$ by $1$.  Hence there is a smallest integer $m \geq 1$ such that either $\nu({\bf v}(-mE))\cdot E \geq -1$ or $\nu({\bf v}(-mE))\cdot F \leq -1$, and by induction we find $h^0(\F_e,\cV)=h^0(\F_e,\cV(-mE))$.   

If $\nu({\bf v}(-mE))\cdot F \leq -1$, Corollary \ref{cor-betti}  gives $h^0(\F_e,\cV)=h^2(\F_e,\cV) = 0$, and so $h^1(\F_e,\cV) = -\chi({\bf v})$.  

On the other hand if $\nu({\bf v}(-mE))\cdot F > -1$, then $\nu({\bf v}(-mE))\cdot E \geq -1$. Corollary \ref{cor-betti} gives $h^0(\F_e,\cV) = \max\{\chi({\bf v}(-mE)),0\}$ and $h^2(\F_e,\cV)=0$, and $h^1(\F_e,\cV)$ is determined by Riemann-Roch.  Note that in this case $h^1(\F_e,\cV)$ is always nonzero.  Indeed, since $m\geq 1$, we find $\nu({\bf v}(-E)) \cdot F >-1$. By Corollary \ref{cor-betti}, we have $h^2(\F_e,\cV(-E))=0$.  Then $H^1(\F_e,\cV)\to H^1(\F_e,\cV|_E)$ is surjective, and $h^1(\F_e,\cV|_E)>0$ since $\nu({\bf v})\cdot E < -1$.
\end{proof}

\subsection{The Brill-Noether theorem}  We call a character ${\bf v}\in K(\F_e)$ of positive rank \emph{special} if the general sheaf $\cV\in \cP_F({\bf v})$ has more than one nonzero cohomology group, and \emph{nonspecial} otherwise.  It is useful for applications to convert the statement of Theorem \ref{thm-betti} to a classification of the nonspecial characters.  Again we concentrate on the case $\nu({\bf v})\cdot F \geq -1$ and rely on Serre duality otherwise.

\begin{corollary}\label{cor-BN}
Let ${\bf v}\in K(\F_e)$ be a character with positive rank and $\Delta({\bf v})\geq 0$, and suppose $\nu({\bf v})\cdot F \geq -1$.  Then ${\bf v}$ is nonspecial if and only if one of the following holds.
\begin{enumerate}
\item We have $\nu({\bf v})\cdot F = -1$.
\item We have $\nu({\bf v})\cdot F > -1$ and $\nu({\bf v})\cdot E \geq -1$.
\item If $\nu({\bf v})\cdot F > -1$ and $\nu({\bf v})\cdot E < -1$, let $m$ be the smallest positive integer such that either $\nu({\bf v}(-mE))\cdot F \leq -1$ or $\nu({\bf v}(-mE))\cdot E \geq -1$.
\begin{enumerate}
\item If $\nu({\bf v}(-mE))\cdot F\leq -1$, then ${\bf v}$ is nonspecial.
\item If $\nu({\bf v}(-mE))\cdot F > -1$, then ${\bf v}$ is nonspecial if and only if $\chi({\bf v}(-mE)) \leq 0$.
\end{enumerate}
\end{enumerate}
\end{corollary}
\begin{proof}
Theorem \ref{thm-betti} shows ${\bf v}$ is nonspecial if we are in case (1) or (2).  In case (3), the proof of Proposition \ref{prop-splitE} shows ${\bf v}$ is nonspecial if and only if the listed conditions are satisfied.
\end{proof}

It is worth pointing out that when $\chi({\bf v})\geq 0$, then case (3) in Corollary \ref{cor-BN} does not occur, so that the classification takes a particularly simple form.   

\begin{corollary}\label{cor-BNpositive}
Let ${\bf v}\in K(\F_e)$ be a character with positive rank and $\Delta({\bf v})\geq 0$, and suppose $\nu({\bf v})\cdot F \geq -1$ and $\chi({\bf v})\geq 0$.  Then ${\bf v}$ is nonspecial if and only if one of the following holds.
\begin{enumerate}
\item We have $\nu({\bf v})\cdot F = -1$.
\item We have $\nu({\bf v})\cdot F > -1$ and $\nu({\bf v})\cdot E \geq -1$.
\end{enumerate}
\end{corollary}
\begin{proof}
Let $\cV\in \cP_F({\bf v})$ be general.  Assume $\nu({\bf v})\cdot F > -1$ and $\nu({\bf v})\cdot E < -1$; we must show ${\bf v}$ is special.  Since $\chi({\bf v})\geq 0$, it will suffice to show $h^1(\F_e,\cV)$ is nonzero.  

The inequality $\nu({\bf v})\cdot E < -1$ gives $h^1(E,\cV|_E)\neq 0$, so if $h^2(\F_e,\cV(-E))=0$, then the surjection $H^1(\F_e,\cV)\to H^1(E,\cV|_E))$ shows $h^1(\F_e,\cV)\neq 0$.
Now if $\nu({\bf v})\cdot F \geq 0$, then $\nu({\bf v}(-E))\cdot F \geq -1$ and $h^2(\F_e,\cV(-E))=0$ follows from Theorem \ref{thm-betti}, completing the proof in this case.

If instead $-1 < \nu({\bf v})\cdot F < 0$, then the assumptions $\chi({\bf v})\geq 0$, and $\Delta({\bf v})\geq 0$ together imply $\nu({\bf v})\cdot E \geq -1$, contradicting our assumption.  Indeed, write $\nu({\bf v}) = \frac{k}{r}E+\frac{l}{r}F$.  By the Riemann-Roch formula,  \begin{align*}\chi({\bf v}) &= r\left(\left( \frac{k}{r} +1 \right)\left( \frac{l}{r} +1 \right)-\frac{ek}{2r} \left( \frac{k}{r} +1 \right) - \Delta({\bf v})\right) \\& = r\left(\left(\frac{k}{r}+1\right)\left(\frac{l}{r}+1-\frac{ek}{2r}\right) -\Delta({\bf v})\right).\end{align*} Then $\nu({\bf v})\cdot F = \frac{k}{r}$, so our assumptions $\nu({\bf v})\cdot F > -1$, $\Delta({\bf v}) \geq 0$, and $\chi({\bf v})\geq 0$ imply $\frac{l}{r} \geq -1 + \frac{ek}{2r}$.  Hence, $$\nu({\bf v})\cdot E = \frac{l}{r}-\frac{ek}{r} \geq -1+\frac{ek}{2r} \geq -1$$ since $\frac{k}{r} = \nu({\bf v})\cdot F < 0$.  Thus this case never arises.
\end{proof}

\section{Gaeta-type resolutions}\label{sec-Gaeta}
In this section, we study resolutions of sheaves on $\F_e$ analogous to the Gaeta resolution of general sheaves on $\P^2$.  Families of such resolutions give unirational parameterizations of moduli spaces of sheaves, and so give an important tool for studying general sheaves.

\begin{definition}
Let $L$ be a line bundle on $\F_e$.  An \emph{$L$-Gaeta-type resolution} of a sheaf $\cV$ on $\F_e$ is a resolution of $\cV$ of the form $$0\to L(-E-(e+1)F)^\alpha \to L(-E-eF)^\beta \oplus L(-F)^\gamma \oplus L^\delta \to \cV\to 0$$ where  $\alpha,\beta,\gamma,\delta$ are nonnegative integers.  We say a sheaf $\cV$ has a \emph{Gaeta-type resolution} if it admits an $L$-Gaeta-type resolution for some line bundle $L$.
\end{definition}

Our main result in this section constructs Gaeta-type resolutions of general prioritary sheaves $\cV\in \cP_F({\bf v})$ on $\F_e$ under mild assumptions on ${\bf v}$.

\begin{theorem}\label{thm-Gaeta}
Let ${\bf v}\in K(\F_e)$ be a Chern character of positive rank, and assume $$\begin{array}{ll} \Delta({\bf v})  \geq 1/4 & \mbox{if $e = 0$}\\
\Delta({\bf v}) \geq 1/8 & \mbox{if $e = 1$}\\
\Delta({\bf v}) \geq 0 & \mbox{if $e\geq 2$}.\\
\end{array}$$ Let $\cV\in \cP_F({\bf v})$ be a general prioritary sheaf. Then $\cV$ admits a Gaeta-type resolution.

In particular, if there are $\mu_H$-semistable sheaves of character ${\bf v}$, the same result holds for a general $\cV\in M_H^{\mu\textit{-ss}}({\bf v})$.  
\end{theorem}

First we observe that if a sheaf $\cV$ admits an $L$-Gaeta-type resolution, then the exponents in the resolution are determined by ${\bf v} = \ch \cV$, and they are easily computable.

\begin{lemma}\label{lem-exponents}
Suppose $\cV$ is a sheaf on $\F_e$ with an $L$-Gaeta-type resolution $$0\to L(-E-(e+1)F)^\alpha\to L(-E-eF)^\beta\oplus L(-F)^\gamma \oplus L^\delta \to \cV\to 0.$$ Then the exponents $\alpha,\beta,\gamma,\delta$ are the integers
\begin{align*}
\alpha &= -\chi(\cV(-L-E-F))\\
\beta &= -\chi(\cV(-L-E))\\
\gamma &= -\chi(\cV(-L-F))\\
\delta &= \chi(\cV(-L))
\end{align*} which depend only on ${\bf v} = \ch \cV$.  
In particular, $L$ must be a line bundle such that the inequalities 
\begin{align}\label{lb-ineqs}\tag{$\dagger$} \begin{split}\chi({\bf v} (-L))& \geq  0 \\ \chi({\bf v}  (-L-E)) &\leq  0 \\ \chi({\bf v}  (-L-F)) &\leq  0 \\ \chi({\bf v} (-L-E-F)) &\leq  0 \end{split}\end{align}
are satisfied.  
\end{lemma}
\begin{proof}
For example, to establish the equality $\alpha = -\chi(\cV(-L-E-F))$ we tensor the resolution of $\cV$ by $\OO_{\F_e}(-L-E-F)$ and use the additivity of the Euler characteristic and Riemann-Roch for line bundles to find \begin{align*}\chi(\cV(-L-E-F)) &= \beta \chi(\OO_{\F_e}(-2E-(e+1)F))\\&\quad +\gamma \chi(\OO_{\F_e}(-E-2F) \\&\quad + \delta \chi(\OO_{\F_e}(-E-F))\\&\quad -\alpha\chi(\OO_{\F_e}(-2E-(e+2)F)\\&=-\alpha \end{align*}
The other equalities are proved in the same way.
\end{proof}

Conversely, if a line bundle $L$ can be found such that the numerical inequalities (\ref{lb-ineqs}) are satisfied, then the next result shows that general sheaves admit $L$-Gaeta-type resolutions.  Thus the problem of constructing Gaeta-type resolutions is reduced to the numerical problem of finding a line bundle $L$ satisfying the inequalities (\ref{lb-ineqs}).

\begin{proposition}\label{prop-constructRes}
Suppose ${\bf v}\in K(\F_e)$ is a Chern character of positive rank and there is a line bundle $L$ such that the inequalities (\ref{lb-ineqs}) are satisfied.  Then $\cP_F({\bf v})$ is nonempty and a general $\cV\in \cP_F({\bf v})$ admits an $L$-Gaeta-type resolution.
\end{proposition}
\begin{proof}
Define positive integers $\alpha,\beta,\gamma,\delta$ as in the statement of Lemma \ref{lem-exponents} and define $$A = L(-E-(e+1)F)^\alpha \qquad B = L(-E-eF)^\beta \oplus L(-F)^\gamma \oplus L^\delta.$$ Then any complex $A\fto{\phi} B$ sitting in degrees $-1$ and $0$ has Chern character ${\bf v}$.  Indeed, $K(\F_e) \te \Q$ has a basis given by the four line bundles \begin{equation*} -L-E-F,-L-E,-L-F,-L.\end{equation*}  If ${\bf w}$ is the character of the complex $A\fto\phi B$, then by the choice of the integers $\alpha,\beta,\gamma,\delta$ we find that ${\bf v} - {\bf w}$ is orthogonal to each of these four line bundles under the Euler pairing $({\bf u},{\bf u'}) = \chi({\bf u}\te {\bf u}')$.  Since the Euler pairing is nondegenerate, ${\bf v} = {\bf w}$.  In particular, the rank of the complex $r(B)-r(A) = r({\bf v})$ is positive.

Let $U \subset \Hom(A,B)$ be the open subset parameterizing injective sheaf maps $\phi$ with torsion-free cokernel $\cV_\phi$.  We check the hypotheses (1)-(3) of Theorem  \ref{thm-resComplete}.  Note that each of the line bundles \begin{align*}\sHom(L(-E-(e+1)F,L(-E-eF)) &= \OO_{\F_e}(F) \\
\sHom(L(-E-(e+1)F),L(-F))&= \OO_{\F_e}( E+eF)\\
\sHom(L(-E-(e+1)F),L) &= \OO_{\F_e}(E+(e+1)F))\end{align*}
are globally generated.  The vanishings $\Ext^1(A,B(-F)) = 0$ and $\Ext^2(B,B(-F))$ both follow immediately from Theorem \ref{thm-lineBundle}.  We already saw $r(B)-r(A) > 0$.  Thus the hypotheses of Theorem \ref{thm-resComplete} are satisfied, and $\cV_\phi/U$ is a nonempty complete family of $F$-prioritary sheaves which admit $L$-Gaeta-type resolutions.  Since $\cP_F({\bf v})$ is irreducible, this completes the proof.
\end{proof}

Finally we show that under the assumptions on $\Delta({\bf v})$ in Theorem \ref{thm-Gaeta} it is possible to find a suitable line bundle $L$.  Together with Proposition \ref{prop-constructRes}, the next result completes the proof of Theorem \ref{thm-Gaeta}.
 
\begin{lemma}\label{lem-Lexists}
Let ${\bf v}\in K(\F_e)$ be a Chern character of positive rank, and assume $$\begin{array}{ll} \Delta({\bf v})  \geq 1/4 & \mbox{if $e = 0$}\\
\Delta({\bf v}) \geq 1/8 & \mbox{if $e = 1$}\\
\Delta({\bf v}) \geq 0 & \mbox{if $e\geq 2$}.\\
\end{array}$$  Then there exists a line bundle $L$ such that the inequalities (\ref{lb-ineqs}) hold.
\end{lemma}
\begin{proof}For simplicity we first assume $\Delta({\bf v}) >0$, and handle the case $\Delta({\bf v}) = 0$ (if $e\geq 2$) later.
Consider the curve  $Q:\chi({\bf v}(-L_{a,b})) = 0$ in the $(a,b)$-plane, where  $L_{a,b}$ is the variable line bundle $$L_{a,b}= \nu({\bf v}) -\frac{1}{2}K_X+ aE +bF \qquad (a,b\in\R)$$ ``centered'' at $\nu({\bf v})-\frac{1}{2}K_X$.  Then by Riemann-Roch, \begin{align*} \frac{\chi({\bf v}(-L_{a,b}))}{r({\bf v})} &= (1-(a+1))(1-(b+1+\frac{e}{2})) \\&\quad - \frac{e(-(a+1))(-(a+1)+1)}{2}-\Delta({\bf v}),\end{align*} so $Q$ is the hyperbola $$\Delta({\bf v}) = a\left(b-\frac{1}{2}ae\right)$$ with asymptotes $$\ell_1:a=0 \qquad \mbox{and} \qquad \ell_2: b = \frac{1}{2}ae$$ meeting at the origin.  Since $\Delta({\bf v}) > 0$, the right branch $Q_1$ (resp. left branch $Q_2$) of $Q$ lies right of $\ell_1$ and above $\ell_2$ (resp. left of $\ell_1$ and below $\ell_2$).  The function $\chi({\bf v}(-L_{a,b}))$ is negative for any $(a,b)$ on $\ell_1$, so it is negative for all points lying between the two branches and positive for all points which are either below $Q_2$ or above $Q_1$.  

Those $(a,b)\in \R^2$ such that $L_{a,b}$ is integral form a shift $\Lambda\subset \R^2$ of the standard integral lattice $\Z^2\subset \R^2$.  Thus to construct the line bundle $L$ we only need to find a point $(a,b)\in \Lambda$ such that $(a,b)$ lies below $Q_2$ and the points $(a+1,b)$, $(a,b+1)$, $(a+1,b+1)$ lie between $Q_2$ and $Q_1$.  It is easy to find a point $(a,b)\in \Lambda$ such that $(a,b)$ lies below $Q_2$ and $(a+1,b)$ and $(a,b+1)$ are both above $Q_2$: start from an arbitrary point $(a,b)\in \Lambda$ below $Q_2$ and repeatedly increment $a$ and/or $b$ by $1$ until increasing either will cross $Q_2$.  Let $(a,b)\in \Lambda$ be any point below $Q_2$ such that $(a+1,b)$ and $(a,b+1)$ are both above $Q_2$.  The point $(a+1,b)$ still lies below $\ell_2$, and the point $(a,b+1)$ still lies left of $\ell_1$, from which we conclude that the points $(a+1,b)$ and $(a,b+1)$ both lie between $Q_2$ and $Q_1$.  For this choice of $(a,b)$ we claim that additionally $(a+1,b+1)$ is between $Q_2$ and $Q_1$.

Since $Q_2$ can be described as the graph of a function $b=f(a)$, the fact that $(a+1,b)$ is above $Q_2$ implies $(a+1,b+1)$ is above $Q_2$.  It remains to show that $(a+1,b+1)$ is below $Q_1$.  If $e\geq 2$, then since $(a,b)$ is below $\ell_2$ and $\ell_2$ has slope $e/2 \geq 1$, we find that $(a+1,b+1)$ is below $\ell_2$ and hence $(a+1,b+1)$ is below $Q_1$.

On the other hand, for $e = 0$ a simple computation shows that the translate of $Q_2$ by the vector $(1,1)$ is disjoint from $Q_1$ if $\Delta({\bf v}) > 1/4$ and tangent to $Q_1$ if $\Delta({\bf v}) = 1/4$.  Thus assuming $\Delta({\bf v}) \geq 1/4$, the point $(a+1,b+1)$ is on or below $Q_1$.  An identical computation shows that an analogous result holds when $e=1$ and $\Delta({\bf v})\geq 1/8$.

If $e \geq 2$ and $\Delta({\bf v}) = 0$, the hyperbola $Q$ degenerates to the union of lines $\ell_1 \cup \ell_2$, and a similar argument works.  Thus a suitable line bundle $L$ can be found in each case.
\end{proof}

\begin{remark}
The stronger inequalities on $\Delta({\bf v})$ needed in case $e=0$ or $e=1$ are in general necessary.  For an easy example, suppose $e=0$, $\nu({\bf v}) = \frac{1}{2}E + \frac{1}{2}F$ (so $r({\bf v})$ is even), and assume $0 < \Delta({\bf v}) < 1/4$, which can certainly be arranged if $r({\bf v})$\ is sufficiently large.  The shifted lattice $\Lambda\subset \R^2$ is $\Z^2 + (\frac{1}{2},\frac{1}{2})$.  The hyperbola $Q$ becomes $ab = \Delta({\bf v})$, so the points $(a,b)\in \Lambda$ such that $\chi(\cV(-L_{a,b}))>0$ are exactly the points of $\Lambda$ in the first and third quadrants.  No choice $(a,b)$ making the inequalities (\ref{lb-ineqs}) hold exists, so no sheaf of character ${\bf v }$ admits a Gaeta-type resolution.
\end{remark}

\section{Globally generated bundles}\label{sec-gg}

Throughout this section we let ${\bf v}\in K(\F_e)$ be a character of positive rank such that $\Delta({\bf v})\geq 0$ and $\nu({\bf v})$ is nef.  In this section we classify those characters ${\bf v}$ such that the general sheaf $\cV\in \cP_F({\bf v})$ is globally generated.  We say that ${\bf v}$ is \emph{globally generated} if a general sheaf $\cV\in \cP_F({\bf v})$ is.  There is no loss of generality in assuming $\nu({\bf v})$ is nef.  Indeed, if $r({\bf v}) \geq 2$ then the general sheaf $\cV\in \cP_F({\bf v})$ is a vector bundle, and if it is globally generated then $c_1({\bf v})$ is nef by Lemma \ref{lem-ggnef}.  In the rank 1 case, if $L$ is a line bundle and $Z\subset X$ is a \emph{general} collection of $n$ points, then $L\te I_Z$ can only be globally generated if $L$ is, and therefore $L$ is again nef by Lemma \ref{lem-ggnef}. 

Since $\nu({\bf v})$ is nef, it follows from Theorem \ref{thm-betti} that a general $\cV\in \cP_F({\bf v})$ is nonspecial.  While the substack of $\cP_F({\bf v})$ of globally generated sheaves is not necessarily open, the substack of $\cP_F({\bf v})$ of globally generated sheaves with no higher cohomology is open.  Thus, a character ${\bf v}$ is globally generated if and only if there exists a sheaf $\cV\in \cP_F({\bf v})$ which is globally generated and has no higher cohomology.  As usual, if ${\bf v}$ is globally generated and there are $\mu_H$-semistable sheaves of character ${\bf v}$, then the general $\cV\in M_H^{\mu{\textit{-ss}}}({\bf v})$ is also globally generated.

The classification of globally generated Chern characters on $\F_e$ falls into three main cases.  First, if $\nu({\bf v})\cdot F = 0$, then we will see that ${\bf v}$ can only be globally generated if it is pulled back from $\P^1$.  This imposes strong restrictions on ${\bf v}$.  When $\nu({\bf v})\cdot F >0$, it is convenient to discuss two separate cases depending on the sign of $\chi({\bf v}(-F))$.  Note that since $\nu({\bf v})$ is nef, we have $\nu({\bf v}(-F)) \cdot F > 0$  and $\nu({\bf v}(-F)) \cdot E\geq -1$. Thus by Theorem \ref{thm-betti}, the cohomology of $\cV(-F)$ is determined by its Euler characteristic.  When $\chi({\bf v}(-F)) \geq 0,$ a general sheaf $\cV$ of character ${\bf v}$ has $H^1(\F_e,\cV(-F)) = 0$ and it is easy to prove global generation by restricting to a fiber.  On the other hand when $\chi({\bf v}(-F))<0$, any sheaf $\cV$ of character ${\bf v}$ has $H^1(\F_e,\cV(-F))\neq 0$ and the restriction to a fiber is not so useful.  In this case we construct globally generated vector bundles by first constructing a suitable ``Lazarsfeld-Mukai'' type bundle $$0 \to \cM \to \OO^{\chi({\bf v})}_{\F_e} \to \cV\to 0.$$ Our Gaeta-type resolutions provide a key tool in analyzing the bundle $\cM$.  

We now state the full classification theorems.  The classification is slightly different depending on if $e=0$, $e=1$, or $e\geq 2$.  To make the statements as clean as possible we state the $e=0$ case separately.

\begin{theorem}\label{thm-gg1}
Suppose $e\geq 1$, and let ${\bf v}\in K(\F_e)$ be a Chern character of positive rank such that $\Delta({\bf v})\geq 0$ and $\nu({\bf v})$ is nef.  Then ${\bf v}$ is globally generated if and only if one of the following holds.
\begin{enumerate}
\item We have $\nu({\bf v})\cdot F = 0$, and there are integers $a,m\geq 0$ such that $${\bf v} = (r({\bf v})-m)\ch \OO_{\F_e}(aF) + m \ch \OO_{\F_e}((a+1)F).$$

\item We have $\nu({\bf v}) \cdot F > 0$ and $\chi({\bf v}(-F)) \geq 0$.

\item We have $\nu({\bf v})\cdot F > 0$, $\chi({\bf v}(-F)) < 0$, and $\chi({\bf v})\geq r({\bf v})+2$.
\item We have $e=1$, $\nu({\bf v})\cdot F >0$, $\chi({\bf v}(-F))<0$, $\chi({\bf v}) = r({\bf v})+1$, and $${\bf v} = (r({\bf v})+1)\ch \OO_{\F_1} - \ch \OO_{\F_1}(-2E-2F).$$
\end{enumerate}
\end{theorem}

The classification changes slightly in the case of $\P^1\times \P^1$ since $F$ and $E$ are the fiber classes for the two rulings. 

\begin{theorem}\label{thm-gg0}
Let ${\bf v}\in K(\P^1\times \P^1)$ be a Chern character of positive rank such that $\Delta({\bf v}) \geq 0$ and $\nu({\bf v})$ is nef.  Let $F_1,F_2$ be fibers in opposite rulings.  Then ${\bf v}$ is globally generated if and only if one of the following holds.
\begin{enumerate}
\item We have $\nu({\bf v})\cdot F_i = 0$ for some $i\in \{1,2\}$, and there are integers $a,m\geq 0$ such that $${\bf v} = (r({\bf v})-m)\ch \OO_{\P^1\times \P^1}(aF_i) + m\ch \OO_{\P^1\times \P^1}((a+1)F_i).$$
\item We have $\nu({\bf v})\cdot F_i > 0$ for $i=1,2$, but $\chi({\bf v}(-F_j)) \geq 0$ for some $j\in \{1,2\}$.
\item We have $\nu({\bf v})\cdot F_i > 0$ and $\chi({\bf v}(-F_i))<0$ for $i=1,2$, and $\chi({\bf v})\geq r({\bf v})+2$.
\end{enumerate}
\end{theorem}

We note that an analogous result for $\P^2$ follows from the classification of globally generated characters on $\F_1$.  Let $H\subset \P^2$ be the class of a line.  We say a character ${\bf v}\in K(\P^2)$ is globally generated if a general $H$-prioritary sheaf is globally generated.  By Hirschowitz-Laszlo \cite{HirschowitzLaszlo} and G\"ottsche-Hirschowitz \cite{GottscheHirschowitz}, the stack $\cP_{\P^2,H}({\bf v})$ is irreducible and a general $\cV\in \cP_{\P^2,H}({\bf v})$ has only one nonzero cohomology group.  This result completes the classification of globally generated characters on $\P^2$ begun in Bertram-Goller-Johnson \cite{BGJ}.

\begin{corollary}\label{cor-P2gg}
Let ${\bf v}\in K(\P^2)$ be a Chern character of positive rank such that  $\Delta({\bf v}) \geq 0$ and $\mu({\bf v}) \geq 0$.  Then ${\bf v}$ is globally generated if and only if one of the following holds.
\begin{enumerate}
\item We have $\mu({\bf v}) = 0$ and ${\bf v} = r({\bf v}) \ch \OO_{\P^2}$.
\item We have $\mu({\bf v}) >0$ and $\chi({\bf v}(-1)) \geq 0$.
\item We have $\mu({\bf v}) >0$, $\chi({\bf v}(-1)) < 0$, and $\chi({\bf v}) \geq r({\bf v}) + 2$.
\item We have $\mu({\bf v}) > 0$, $\chi({\bf v}(-1)) < 0$, $\chi({\bf v}) = r({\bf v})+1$, and $${\bf v} = (r({\bf v})+1)\ch \OO_{\P^2}-\ch \OO_{\P^2}(-2).$$
\end{enumerate}
\end{corollary}
\begin{proof}
Let $\pi :\FF_1\to \P^2$ be the blowdown map, and let ${\bf w} = \pi^*({\bf v}) \in K(\FF_1)$ and $r = r({\bf v}) = r({\bf w})$.  The result is clear if $r=1$, since then the pullback of a general sheaf in $\cP_{\P^2,H}({\bf v})$ is a general sheaf in $\cP_{\FF_1,F}({\bf v})$, and clearly ${\bf v}$ is globally generated if and only if ${\bf w}$ is.  So suppose $r\geq 2$.

The main technical difficulty is to compare the notions of $F$-prioritary sheaves on $\F_1$ and $H$-prioritary sheaves on $\P^2$.   We have $c_1({\bf w})\cdot E = 0$ and $\Delta({\bf w}) = \Delta({\bf v}) \geq 0$, so by Corollary \ref{cor-balanced} a general $\cW\in \cP_{\F_1,F}({\bf w})$ restricts to a trivial bundle on $E$: $\cW|_E \cong \OO_E^r$.  Furthermore, since $\cW$ is general it is actually $(E+F)$-prioritary.  This can be shown by an argument similar to the proof of Lemma \ref{lem-ds}: since $c_1({\bf w}) \cdot E = 0$, we can construct a direct sum of line bundles $$\cW' = \OO_{\F_1}(mE+mF)^a \oplus \OO_{\F_1}((m+1)E+(m+1)F)^b$$ having the same rank and $c_1$ as ${\bf w}$.  Then we compute $$2 r^2 \Delta({\cW}') = -ab \leq 0.$$ Since $\Delta({\bf w}) \geq 0$, we can obtain a sheaf of character ${\bf w}$ from $\cW'$ by repeated elementary modifications.  Since $\cW'$ is clearly $(E+F)$-prioritary, so is the general $\cW\in \cP_{\F_1,F}({\bf w})$ by Lemma \ref{lem-elementaryprioritary}.

Furthermore, $(E+F)$-prioritary vector bundles on $\F_1$ are automatically $F$-prioritary.  Indeed, if $\cW$ is an $(E+F)$-prioritary vector bundle, then the sequence $$0\to \cW(-E-F)\to \cW(-F)\to \cW(-F)|_E\to 0$$ yields $$\Ext^2(\cW,\cW(-E-F))\to \Ext^2(\cW,\cW(-F))\to \Ext^2(\cW,\cW(-F)|_E).$$ The first group vanishes by assumption, and the last vanishes since it is the $H^2$ of a sheaf supported on a curve (as $\cW$ is locally free).  Therefore $\cW$ is $F$-prioritary.

Now let $\cP' \subset \cP_{\F_1,F}({\bf w})$ be the open dense substack parameterizing $(E+F)$-prioritary vector bundles with restriction $\cW|_E \cong \OO_E^r$, and let $\cP'' \subset \cP_{\P^2,H}({\bf v})$ be the open dense substack parameterizing vector bundles.  If $\cW\in \cP'$, then $\pi_* \cW$ is locally free and $\pi^*\pi_* \cW \cong \cW$ by Walter \cite[Lemma 6]{Walter}.  Furthermore, we have an isomorphism $\Ext^2(\cW,\cW(-E-F)) \cong \Ext^2(\pi_*\cW,\pi_*\cW(-H))$, which shows that $\pi_* \cW$ is $H$-prioritary.  Thus there is an induced map $\pi_* : \cP' \to \cP''$.  On the other hand, if $\cV\in \cP''$ then $\pi_*\pi^* \cV \cong \cV$ since $\cV$ is a vector bundle and, furthermore, $\pi^*\cV$ is $(E+F)$-prioritary.  Therefore $\pi^* \cV \in \cP'$ and pullback gives an inverse map $\pi^* : \cP''\to \cP'$.
 
 Under the correspondence between $\cP'$ and $\cP''$, globally generated bundles correspond to globally generated bundles.  We conclude that ${\bf v}$ is globally generated if and only if ${\bf w}$ is globally generated, and the result follows from Theorem \ref{thm-gg1}.
 \end{proof}

Note that Corollary \ref{cor-P2gg} can also be proved more directly by mimicking the proof of Theorem \ref{thm-gg1}. We discuss each of the three main cases (1)-(3) of the classification in its own subsection.

\subsection{Pullbacks}  We begin the proof of the classification by analyzing the case $\nu({\bf v})\cdot F =0$.  In this case a globally generated bundle must be a pullback from $\P^1$, which imposes strong restrictions on the character.

\begin{proposition}\label{prop-ggStrictNef}
Suppose $\nu({\bf v})\cdot F =0$.  Then ${\bf v}$ is globally generated if and only if ${\bf v}$ is of the form $${\bf v} = (r({\bf v})-m)\ch \OO_{\F_e}(aF) + m\ch \OO_{\F_e}((a+1)F)$$ for some integers $a,m\geq 0$.
\end{proposition} 
\begin{proof}
$(\Leftarrow)$ The bundles $$\cV = \OO_{\F_e}(aF)^{r-m} \oplus \OO_{\F_e}((a+1)F)^m$$ are $F$-prioritary, globally generated, and have no higher cohomology, so their Chern characters are globally generated.

$(\Rightarrow)$ Suppose $\cV\in \cP_F({\bf v})$ is general and globally generated.  The result is clear if $r({\bf v}) = 1$, so suppose $r({\bf v}) \geq 2$.  Then $\cV$ is a vector bundle. 
  The restriction $\cV|_{F}$ to any fiber $F\cong \P^1$ of $\pi:\F_e \to \P^1$ has degree $0$.  If any factor of $\cV|_F$ has negative degree, then $\cV|_F$ is not globally generated, a contradiction.  Therefore $\cV|_F \cong \OO_F^{r({\bf v})}$ is trivial on each fiber.
 Consider the exact sequence $$0 \longrightarrow \cV(-E) \longrightarrow \cV \longrightarrow \cV|_E \longrightarrow 0$$ and apply $\pi_*$. Since $\cV(-E)|_F \cong \OO_F(-1)^{r({\bf v})}$ for any fiber $F$, we conclude that $\pi_* \cV(-E) = R^1 \pi_* \cV(-E) =0$ and $\pi_* \cV \cong \pi_* (\cV|_E)$.  Hence, by \cite[Lemma 5]{Walter},  $\cV \cong \pi^*(\pi_* (\cV|_E))$.  Since $\cV|_E$ is balanced by Corollary \ref{cor-balanced}, it follows that $$\cV\cong \OO_{\F_e}(aF)^{r({\bf v})-m} \oplus \OO_{\F_e}((a+1)F)^m$$ for some integers $m\geq 0$ and $a\in \Z$.  As $\cV$ is globally generated, $a\geq 0$.
\end{proof}

\subsection{Restriction to a fiber} The case where $\chi({\bf v}(-F)) \geq 0$ and $\nu({\bf v})\cdot F > 0$ is the simplest to analyze.

\begin{proposition}\label{prop-ggPositive}
Suppose $\chi({\bf v}(-F))\geq 0$ and $\nu({\bf v})\cdot F >0$.  Then ${\bf v}$ is globally generated.
\end{proposition}
\begin{proof}
First suppose $r({\bf v})\geq 2$, so that the general $\cV\in \cP_F({\bf v})$ is a vector bundle.  Since $\nu({\bf v})\cdot F >0$, the bundle $\cV$ restricts to a globally generated vector bundle on every fiber $F$ of the projection $\pi:\F_e\to \P^1$ by \cite[Lemmas 3 and 4]{Walter}.  Let $\cV$ be such a bundle which also has no higher cohomology. Let $p\in \F_e$ be any point, and let $F_p:=\pi^{-1}(\pi(p))$ be the    fiber through $p$.  Since $\chi({\bf v}(-F))\geq 0$, as explained in the introduction to this section, we know that $H^1(\F_e,\cV(-F))=0$.  Then the restriction sequence $$0\to \cV(-F)\to \cV \to \cV|_{F_p}\to 0$$ shows that $H^0(\F_e,\cV)\to H^0(F_p,\cV|_{F_p})$ is surjective.  Since $\cV|_{F_p}$ is globally generated, this implies that $\cV$ is globally generated at every point of $F_p$.  Therefore $\cV$ is globally generated, and so is ${\bf v}$.

If $r({\bf v}) = 1$, suppose the general sheaf of character ${\bf v}$ is of the form $L \te I_Z$, where $L$ is a nef line bundle and $Z\subset X$ is a collection of $n$ general points.  For $p\in \F_e$ we have an exact sequence of one of the following two forms, depending on whether or not the fiber $F_p$ contains a point of $Z$:
$$0\to L(-F) \te I_Z \to L \te I_Z \to L|_{F_p} \to 0 $$
$$0\to L(-F)\te I_{Z'} \to L \te I_Z \to L|_{F_p}(-1) \to 0.$$
Here $Z' \subset Z$ consists of $n-1$ points, with the $n$th point of $Z$ lying on $F_p$.  Note that both $L(-F)\te I_Z$ and $L(-F)\te I_{Z'}$ have no higher cohomology, and $L|_{F_p}$ and $L|_{F_p}(-1)$ are both globally generated.  Therefore, $L\te I_Z$ is globally generated except possibly at points in $Z$.  So suppose $p\in Z$.  In this case, we can construct two curves in $\FF_e$:
\begin{enumerate} \item There is a curve $C$ of class $L$ that contains $Z$ and intersects $F_p$ transversely.
\item Since $\chi(L(-F)\te I_{Z'})>0$, there is a curve $D$ of class $L(-F)$ which contains $Z'$ and does not contain $p$.  (Note that the points in $Z$ are general, so impose the expected number of conditions on sections of $L(-F)$.)
\end{enumerate}
Then the curves $C$ and $D+F$ give two sections of $L$ which contain $Z$ and intersect transversely at $p$.  Therefore $L\te I_Z$ is globally generated at $p$.
\end{proof}

\begin{remark}
When $e=0$, we find by symmetry that results analogous to Propositions \ref{prop-ggStrictNef} and \ref{prop-ggPositive} hold for the opposite ruling $E$.  So, going forward, we may assume $\chi({\bf v}(-E))<0$ and $\nu({\bf v})\cdot E >0$ when $e=0$.
\end{remark}

\subsection{Lazarsfeld-Mukai bundles} It remains to classify globally generated characters ${\bf v}$ with $\chi({\bf v}(-F)) < 0$ and $\nu({\bf v})\cdot F > 0$.  Throughout this section, we put $${\bf m} = \chi({\bf v})\ch(\OO_{\F_e}) -{\bf v}.$$  Thus, if ${\bf v}$ is globally generated, then ${\bf m}$ is the character of the Lazarsfeld-Mukai bundle $$0\to \cM \to \OO_{\F_e}^{\chi({\bf v})} \to \cV \to 0$$ which is the kernel of the canonical evaluation map.  The next lemma is the main tool we use in this case.

\begin{lemma}\label{lem-mukai}
Suppose $\chi({\bf v}(-F)) < 0$.  Then ${\bf v}$ is globally generated if and only if there is a vector bundle $\cM$ of character ${\bf m}$ such that 
\begin{enumerate}
\item $\cM$ has no cohomology,
\item $h^1(\F_e,\cM(-F))=0$, and 
\item$\cM^*$ is globally generated.
\end{enumerate}
\end{lemma}
\begin{proof}
($\Rightarrow$) Suppose ${\bf v}$ is globally generated.  Then there is a globally generated (torsion-free) sheaf $\cV\in \cP_F({\bf v})$ which has no higher cohomology.  Furthermore, since $\chi(\cV(-F))<0$ we can assume $\cV(-F)$ only has $h^1$.  Then the kernel $\cM$ of the canonical evaluation map $$0\to \cM \to \OO_{\F_e}^{\chi({\bf v})} \to \cV\to 0$$ has the required properties.

($\Leftarrow$) Conversely, suppose there is a vector bundle $\cM$ of character ${\bf m}$ with properties (1)-(3).  Let $\cV$ be a sheaf defined as the cokernel of a general map $\phi:\cM\to \OO_{\F_e}^{\chi({\bf v})}$.  Since $r({\bf v}) \geq 1$ and $\cM^*$ is globally generated, the map $\phi$ is injective and $\cV$ is torsion-free (see \cite[Proposition 2.6]{HuizengaJAG}, and the proof of Theorem \ref{thm-resComplete}).  The sequence $$0\to \cM \fto\phi \OO_{\F_e}^{\chi({\bf v})}\to \cV\to 0$$ shows that $\cV$ is globally generated, and $\cV$ has no higher cohomology since $\cM$\ and $\OO_{\F_e}$ have no higher cohomology.

Finally, we must check that $\cV$ is $F$-prioritary.  But $\Ext^2(\cV,\cV(-F))$ is a quotient of a direct sum of copies of $$\Ext^2(\cV,\OO_{\F_e}(-F)) \cong \Hom(\OO_{\F_e}(-F),\cV(K_{\F_e}))^*\cong H^0(\F_e,\cV(K_{\F_e}+F))^*.$$ There is an injection $\cV(K_{\F_e}+F)\to \cV(-F)$ since $-F-(K_{\F_e}+F) = 2E +eF$ is effective, and $H^0(\F_e,\cV(-F))=0$ since we are assuming $H^1(\F_e,\cM(-F))=0$.  Therefore $H^0(\F_e,\cV(K_{\F_e}+F))=0$ and $\Ext^2(\cV,\cV(-F))=0$.
\end{proof}

The remainder of the classification of globally generated characters ${\bf v}$ with $ \chi({\bf v}) \leq r({\bf v}) + 1$ follows quickly; there is only one such character that was not already studied in the previous subsections.

\begin{corollary}
Suppose $\chi({\bf v}(-F)) < 0$, $\nu({\bf v})\cdot F > 0$, and $\chi({\bf v}) \leq r({\bf v})+1$.  If $e=0$, then further assume $\chi({\bf v}(-E)) <0$ and $\nu({\bf v})\cdot E > 0$.  Then ${\bf v}$ is globally generated if and only if $e=1$ and $${\bf v} = (r({\bf v})+1)\ch\OO_{\F_1}-\ch \OO_{\F_1}(-2E-2F).$$
\end{corollary}
\begin{proof}
($\Rightarrow$)  Suppose ${\bf v}$ is globally generated, and let $\cV$ be a globally generated sheaf of character ${\bf v}$ with no higher cohomology.  Then $\chi({\bf v}) = h^0(\F_e,\cV)$, so $\chi({\bf v}) \geq r({\bf v})$.  If $\chi({\bf v}) = r({\bf v})$ then we must have $\cV \cong \OO_{\F_e}^{r({\bf v})}$, but then $\nu({\bf v})\cdot F = 0$.  So, going forward we may assume $\chi({\bf v}) = r({\bf v})+1$.   By Lemma \ref{lem-mukai}, the kernel $\cM$ of a general evaluation map $$0\to \cM\to \OO_{\F_e}^{\chi({\cV})}\to \cV \to 0$$ is a line bundle with no cohomology such that $H^1(\F_e,\cM(-F))=0$ and $\cM^*$ is globally generated.  Suppose $\cM = \OO_{\F_e}(-aE-bF)$.  Since $\cM^*$ is globally generated, we have $a \geq 0$ and $b\geq ae$.  Since $\nu(\cV)\cdot F > 0$, we further have $a > 0$.  If $a = 1$, then $\chi(\cM(-F)) = \chi(\OO_{\F_e}(-F)) = 0$ and so $\chi(\cV(-F)) = 0$, a contradiction.  Therefore $a \geq 2$ and so $b\geq 2e$.  Then $$H^2(\F_e,\cM) \cong H^0(\F_e,\cM^*(K_{\F_e}))=H^0(\F_e,\OO_{\F_e}((a-2)E+(b-e-2)F)),$$ and since $\cM$ has no cohomology we must have $b < e+2$.  But $b\geq 2e$ and $b < e+2$ together imply $e\in \{0,1\}$.  If $e=0$, then since $a\geq 2$ we find by symmetry that we must also have $b\geq 2$.  But then $\cM$ has $h^2$, so there is no suitable $\cM$.  If $e=1$, then the only possibility is $a=b=2$.

($\Leftarrow$) Conversely, for any integer $r \geq 1$ the cokernel $\cV$ of a general injection $$0\to \OO_{\F_1}(-2E-2F)\to \OO_{\F_1}^{r+1}\to \cV\to 0$$ is a globally generated torsion-free sheafe with no higher cohomology, and ${\bf v}$ satisfies the necessary hypotheses.
\end{proof}

Finally, we construct bundles $\cM$ with the necessary properties in the remaining cases.  We write ${\bf m}^D$ for the Serre dual character ${\bf m}^* (K_X)$.  The next lemma will allow us to apply our theory of Gaeta-type resolutions to study bundles of character ${\bf m}$.

\begin{lemma}\label{lem-serreDualIneqs}
Suppose $\chi({\bf v}(-F)) < 0$, $\nu({\bf v})\cdot F < 0$, and $\chi({\bf v}) \geq r({\bf v})$.  If $e=0$, then further suppose $\chi({\bf v}(-E)) < 0$ and $\nu({\bf v})\cdot E <0$.  The character ${\bf m}^D$ satisfies
\begin{align*}
\chi({\bf m}^D) &= 0\\
\chi({\bf m}^D(-F)) & < 0\\
\chi({\bf m}^D(-E)) & \leq 0
\end{align*}
and, if $e\geq 2$, then furthermore 
$$\chi({\bf m}^D(-E-F) )  < 0.$$

\end{lemma}
\begin{proof}
By Serre duality it is equivalent to show
\begin{align*}
\chi({\bf m}) &= 0\\
\chi({\bf m}(F)) & < 0\\
\chi({\bf m}(E)) & \leq 0\\
\chi({\bf m}(E+F) ) & < 0 \mbox{ (if $e\geq 2$)} 
\end{align*}
The first statement is clear.  For the second, we use Lemma \ref{lem-eulerchar} to compute 
\begin{align*}
\chi({\bf m}(F)) &= \chi({\bf v})\chi(\OO_{\F_e}(F))-\chi({\bf v}(F))\\
&= 2\chi({\bf v}) - (\chi({\bf v})+c_1({\bf v})\cdot F + r({\bf v})(2-1))\\
&= \chi({\bf v})-r({\bf v}) - c_1({\bf v})\cdot F\\
&= \chi({\bf v}(-F)) \\&< 0.
\end{align*}
If $e=0$, then the third inequality follows by symmetry.  If instead $e \geq 1$, then 
\begin{align*}
\chi({\bf m}(E)) &= \chi({\bf v})\chi(\OO_{\F_e}(E))-\chi({\bf v}(E))\\
&= \chi({\bf v})(2-e)-(\chi({\bf v})+c_1({\bf v})\cdot E+r({\bf v})(2-e-1))\\ &= (\chi({\bf v})-r({\bf v}))(1-e)-c_1({\bf v})\cdot E\\ &\leq 0
\end{align*}
since $\chi({\bf v}) \geq r({\bf v})$ and $c_1({\bf v})$ is nef.

Finally, we similarly compute 
\begin{align*}\chi({\bf m}(E+F)) &= \chi({\bf v})\chi(\OO_{\F_e}(E+F)) - \chi({\bf v}(E+F))\\
&=  \chi({\bf v})(4-e) - (\chi({\bf v}) + c_1({\bf v})\cdot (E+F) +r({\bf v})(4-e-1))\\
&= (\chi({\bf v})-r({\bf v}))(3-e)-c_1({\bf v})\cdot (E+F).
\end{align*}
Since $\chi({\bf v}) \geq r({\bf v})$ and $c_1({\bf v})$ is nef with $c_1({\bf v})\cdot F > 0$, it follows that $\chi({\bf m}(E+F)) < 0$ if $e\geq 3$.  Suppose $e=2$.  Then the above expression reduces to $$\chi({\bf m}(E+F)) = \chi({\bf v})-r({\bf v}) - c_1({\bf v})\cdot (E+F) = \chi({\bf v}(-E-F)).$$ Let $\cV\in \cP_F({\bf v})$ be general.  Then  $\cV(-F)$ only has $h^1$, and $\cV(-F)|_E$ splits as a direct sum of line bundles of degree $\geq -1$.  Then the restriction sequence $$0\to \cV(-E-F)\to \cV(-F)\to \cV(-F)|_E\to 0$$ shows that the only nonzero cohomology of $\cV(-E-F)$ is also $h^1(\F_e,\cV(-E-F))\neq 0$, and therefore $\chi({\bf v}(-E-F)) <0$ as required. 
\end{proof}

Together with the other results in this section, the next result completes the classification and proves Theorems \ref{thm-gg1} and \ref{thm-gg0}.

\begin{proposition}
Suppose $\chi({\bf v}(-F)) < 0$, $\nu({\bf v})\cdot F > 0$, and $\chi({\bf v}) \geq r({\bf v})+2$.  If $e=0$, further assume $\chi({\bf v}(-E)) < 0$ and $\nu({\bf v})\cdot E > 0$.  Then ${\bf v}$ is globally generated.
\end{proposition}
\begin{proof}
First assume $\chi({\bf m}^D(-E-F)) < 0$.  Then by Lemma \ref{lem-serreDualIneqs}, $\OO_{\F_e}$ is a line bundle satisfying the inequalities (\ref{lb-ineqs}) for the character ${\bf m}^D$.  Therefore by Proposition \ref{prop-constructRes}, the stack $\cP_{F}({\bf m}^D)$ is nonempty and a general $\cM^D\in \cP_F({\bf m}^D)$ admits a resolution of the form $$0\to \OO_{\F_e}(-E-(e+1)F)^{\alpha} \to \OO_{\F_e} ( - E- eF)^\beta \oplus \OO_{\F_e}(-F)^\gamma  \to \cM^D\to 0;$$ there are no copies of $\OO_{\F_e}$ in the resolution since $\chi({\bf m}^D) = 0$.  Since $\chi({\bf v})\geq r({\bf v})+2$, we have $r(\cM^D) \geq 2$ and therefore, by Theorem \ref{thm-Walter},  $\cM^D$ is a vector bundle.  Clearly $\cM^D$ has no cohomology, so its Serre dual $\cM$ also has no cohomology.  The bundle $\cM$ fits in a sequence $$0\to \cM \to \OO_{\F_e}(-E-2F)^\beta \oplus \OO_{\F_e}(-2E-(e+1)F)^\gamma \to \OO_{\F_e}(-E-F)^\alpha \to 0,$$ from which it immediately follows that $h^1(\F_e,\cM(-F)) = 0$.  The dual $\cM^*$ has a resolution $$0\to \OO_{\F_e}(E+F)^\alpha \to \OO_{\F_e}(E+2F)^\beta \oplus \OO_{\F_e}(2E + (e+1)F)^\gamma \to \cM^*\to 0.$$ The line bundles $\OO_{\F_e}(E+2F)$ and $\OO_{\F_e}(2E+(e+1)F)$ are each globally generated at all points $p\in \F_e$ with $p\notin E$, and therefore $\cM^*$ is globally generated away from $E$.

To see that $\cM^*$ is globally generated at all points on $E$, observe that $c_1(\cM^*) = c_1({\bf v})$ is nef, so in particular $c_1(\cM^*)\cdot E \geq 0$.  The resolution of $\cM^*$ shows that $h^1(\F_e,\cM^*(-E))=0$.  If $\cM^*|_E$  splits as a balanced direct sum of line bundles, then $\cM^*|_E$ is globally generated and the restriction sequence $$0\to \cM^*(-E) \to \cM^* \to \cM^*|_E\to 0$$ shows that $\cM^*$ is globally generated on $E$.  By Proposition \ref{prop-balanced}, to see $\cM^*|_E$ is balanced (for a general $\cM^*$) it is enough to show $\cM^D$ (and hence $\cM^*$) is $E$-prioritary.  To see $\Ext^2(\cM^D,\cM^D(-E))=0,$ it is enough to verify $\Ext^2(\cM^D,\OO_{\F_e}(-2E-eF)) = \Ext^2(\cM^D,\OO_{\F_e}(-E-F))=0$.  Equivalently by Serre duality, we need $$\Hom(\OO_{\F_e}(-2E-eF),\cM^D(K_{\F_E})) = H^0(\F_e,\cM^D(-2F))=0$$ and $$\Hom(\OO_{\F_e}(-E-F),\cM^D(K_{\F_e}))=H^0(\F_e,\cM^D(-E-(e+1)F))=0.$$ Both vanishings follow immediately from the resolution of $\cM^D$.  Therefore $\cM^*$ is globally generated, and Lemma \ref{lem-mukai} completes the proof in case $\chi({\bf m}^D(-E-F))<0$.

Finally suppose $\chi({\bf m}^D(-E-F)) \geq 0$.  By Lemma \ref{lem-serreDualIneqs}, we must have $e\in \{0,1\}$.  Define integers \begin{align*}
\alpha &= -\chi({\bf m}^D(-E-F))\\
\beta &= -\chi({\bf m}^D(-E))\\
\gamma &= -\chi({\bf m}^D(-F)).
\end{align*}
By Lemma, \ref{lem-serreDualIneqs}, we have $\alpha \leq 0$, $\beta \geq 0$, and $\gamma > 0$.  Then the direct sum of line bundles $$\cM^D := \OO_{\F_e}(-E-(e+1)F)^{-\alpha} \oplus \OO_{\F_e}(-E-eF)^\beta \oplus \OO_{\F_e}(-F)^\gamma$$ has $\ch(\cM^D) = {\bf m}^D$ by an argument analogous to the first part of the proof of Proposition \ref{prop-constructRes}.  Its Serre dual is the bundle $$\cM = \OO_{\F_e}(-E-F)^{-\alpha} \oplus \OO_{\F_e}(-E-2F)^{\beta} \oplus \OO_{\F_e}(-2E-(e+1)F)^\gamma.$$ Then $\cM$ has no cohomology, $h^1(\F_e,\cM(-F))=0$, and $\cM^*$ is globally generated (since $e\leq 1$).  By Lemma \ref{lem-mukai}, ${\bf v}$ is globally generated.
\end{proof}

\subsection{Notes on ampleness}  We close the paper with some remarks on the question that initially led us to study globally generated vector bundles.  Let $X$ be a smooth surface.  Recall that a vector bundle $\cV$ on $X$ is \emph{ample} if the line bundle $\OO_{\PP \cV}(1)$ is ample.  
 
\begin{problem}
Classify the Chern characters of ample vector bundles on $X$.
\end{problem}

On $\PP^2$ and $\FF_e$, ample line bundles are globally generated. In contrast, examples of Gieseker show that higher rank ample vector bundles need not have any sections. For example, a bundle $\cV$ defined by a general resolution of the form $$0 \rightarrow \OO_{\PP^2}(-d)^2 \rightarrow \OO_{\PP^2}(-1)^4 \rightarrow \cV \rightarrow 0$$ is ample provided $d \gg 0$ (see \cite[Example 6.3.17]{Lazarsfeld} or \cite{Gieseker}). However, if a vector bundle $\cV$ is ample, then $\Sym^k \cV$ has no higher cohomology and is globally generated for sufficiently large $k$. By Riemann-Roch this implies the necessary inequality (see also \cite{FultonLazarsfeld})
\begin{equation}\label{ampleinequality}\tag{$\ast$}
\frac{\nu(\cV)^2}{2} > \frac{\Delta(\cV)}{r(\cV)+1}.
\end{equation}
Furthermore, $\nu(\cV)$ needs to be ample and its restriction to any curve needs to be ample. On $\FF_e$, this implies that  $\nu({\cV}) \cdot E \geq 1$ and  $\nu(\cV)\cdot F \geq 1.$ If $\nu(\cV) \cdot F=1$, the restriction of the bundle to every fiber must be $\OO_{\PP^1}(1)^r$, and hence $\cV(-E)$ is pulled back from $\P^1$.  One can ask whether the inequality (\ref{ampleinequality}) suffices to show a general $\cV$ is ample if $\nu(\cV)$ is sufficiently ample. For the Gieseker example above, the inequality (\ref{ampleinequality}) implies that if $\cV$ is ample, then $d \geq 7$. The authors do not know the optimal value of $d$ for $\cV$ to be ample even in this case.

By the next simple observation, our characterization of globally generated characters on $\F_e$ yields sufficient conditions for a character to be the character of an ample bundle.  However, analogues of Gieseker's example (see \cite[Theorem 6.3.65]{Lazarsfeld}) show that these conditions are certainly not necessary.

\begin{lemma}\label{lem-ample}
Let $X$ be a projective variety with an ample divisor $H$.  Suppose ${\cV}$ is a vector bundle on $X$ such that $\cV(-H)$ is globally generated.  Then $\cV$ is ample.
\end{lemma}
\begin{proof}
As $\cV(-H)$ is a quotient of $\OO_X^n$, we find that $\cV$ is the quotient of an ample bundle.
\end{proof}

On $\F_e$, we take $H = E+ (e+1)F$ and deduce the following result.

\begin{corollary}
Suppose $e \geq 2$, and let ${\bf v}\in K(\F_e)$ be a Chern character such that $r({\bf v}) \geq 2$, $\Delta({\bf v})\geq 0$, $\nu({\bf v})\cdot F \geq 1$, and $\nu({\bf v})\cdot E \geq 1$. Then the general $\cV\in \cP_F({\bf v})$ is ample whenever any of the following conditions holds.
\begin{enumerate}
\item We have $\nu({\bf v})\cdot F = 1$, and there are integers $a,m\geq 0$ such that $${\bf v} = (r({\bf v})-m)\ch \OO_{\F_e}(E+(e+a+1)F) + m \ch \OO_{\F_e}(E+(e+a+2)F).$$

\item We have $\nu({\bf v})\cdot F > 1$ and $\chi({\bf v}(-E-(e+2)F)) \geq 0$.

\item We have $\nu({\bf v})\cdot F > 1$, $\chi({\bf v}(-E-(e+2)F)) < 0$, and $\chi({\bf v}(-E-(e+1)F)) \geq r({\bf v})+2$.
\end{enumerate}
\end{corollary}

We leave the analogous statements for $\P^1\times \P^1$, $\F_1$, and $\P^2$ to the reader.  Completing the classification of characters of ample vector bundles remains a very interesting open question for any of these surfaces.

\bibliographystyle{plain}

\begin{thebibliography}{CHW14}



\bibitem[Bea83]{Beauville}
A. Beauville. Complex algebraic surfaces, volume 68 of London Mathematical Society Lecture Note Series. Cambridge University Press, Cambridge, 1983.


 \bibitem[BGJ16]{BGJ}
A. Bertram, T. Goller and D. Johnson, Le Potier's strange duality, quot schemes and multiple point formulas for del Pezzo surfaces, preprint. 


\bibitem[Cos06]{CoskunScroll}
I. Coskun, Degenerations of surface scrolls and the Gromov-Witten invariants of Grassmannians, J. Algebraic Geom., {\bf 15} (2006),  223--284.


\bibitem[CH15]{CoskunHuizengaGokova} I. Coskun and J. Huizenga, The birational geometry of the moduli spaces of sheaves on $\mathbb{P}^2$, Proceedings of the G\"okova Geometry-Topology Conference 2014, (2015), 114--155.

\bibitem[CH16]{CoskunHuizengaWBN}  I. Coskun and J. Huizenga, Weak Brill-Noether for rational surfaces, to appear in Contemp. Math.

\bibitem[Eis05]{Eisenbudsyzygies}
D. Eisenbud. {\em The geometry of syzygies}. Springer, 2005

\bibitem[FL83]{FultonLazarsfeld}
W. Fulton and R. Lazarsfeld,  Positive Polynomials for Ample Vector Bundles, Ann. Math., {\bf 188} no. 1 (1983), 35--60.

\bibitem[Gae51]{Gaeta}
F. Gaeta, Sur la distribution des degr\'es des formes appartenant \`a la matrice de l'id\'eal homog\`ene attach\'e \`a un groupe de $N$ points g\'en\'eriques du plan, C. R. Acad. Sci. Paris {\bf 233} (1951), 912--913.

\bibitem[Gie71]{Gieseker}
D. Gieseker, $p$-ample bundles and their Chern classes, Nagoya Math. J. {\bf 43} (1971), 91--116.

\bibitem[GHi94]{GottscheHirschowitz}
L. G\"{o}ttsche and A. Hirschowitz, Weak Brill-Noether for vector bundles on the projective plane, in Algebraic
geometry (Catania, 1993/Barcelona, 1994), 63--74, Lecture Notes in Pure and Appl. Math., 200 Dekker, New York.


\bibitem[Hart77]{Hartshorne}
R. Hartshorne. Algebraic geometry. Springer-Verlag, New York, 1977. Graduate Texts in Mathematics, No. 52.

 

\bibitem[HiL93]{HirschowitzLaszlo}
A. Hirschowitz\ and\ Y. Laszlo, Fibr\'es g\'en\'eriques sur le plan projectif, Math. Ann. {\bf 297} (1993), no.~1, 85--102.


\bibitem[Hui16]{HuizengaJAG}
J.~Huizenga, Effective divisors on the Hilbert scheme of points in the plane and interpolation for stable bundles, J. Algebraic Geom. {\bf 25} (2016), no.~1, 19--75.


\bibitem[HuL10]{HuybrechtsLehn} D. Huybrechts\ and\ M. Lehn, {\it The geometry of moduli spaces of sheaves}, second edition, Cambridge Mathematical Library, Cambridge Univ. Press, Cambridge, 2010.



\bibitem[Laz04]{Lazarsfeld}
R. Lazarsfeld, Positivity in Algebraic Geometry II, Springer 2004. 



\bibitem[LeP97]{LePotier}
J. Le Potier, {\it Lectures on vector bundles}, translated by A. Maciocia, Cambridge Studies in Advanced Mathematics, 54, Cambridge Univ. Press, Cambridge, 1997. 



\bibitem[Wal98]{Walter}
C. Walter, Irreducibility of moduli spaces of vector bundles on birationally ruled surfaces, 
Algebraic Geometry (Catania, 1993/Barcelona, 1994), Lecture Notes in Pure and Appl. Math., {\bf 200} (1998), 201--211.

\end{thebibliography}

\end{document}